\documentclass{amsart}
\usepackage[utf8]{inputenc}
\usepackage{amsmath}
\usepackage{amsfonts}
\usepackage{amssymb}
\usepackage{amsthm}
\newtheorem{thm}{Theorem}[section]

\newtheorem{con}[thm]{Conjecture}
\newtheorem{lem}[thm]{Lemma}

\newtheorem{rem}[thm]{Remark}
\newcommand{\Pro}{{\mathbb{P}}}

\newcommand{\Real}{{\mathbb{R}}}

\newcommand{\Zet}{{\mathbb{Z}}}
\newcommand{\Nat}{{\mathbb{Z}^+}}
\newcommand{\Natz}{{\mathbb{Z}_0^+}}

\newcommand{\Fl}{{\mathrm{Fl}}}
\newcommand{\ab}{{\mathbf{a}}}
\newcommand{\bb}{{\mathbf{b}}}

\newcommand{\zb}{{\mathbf{z}}}
\newcommand{\yb}{{\mathbf{y}}}
\newcommand{\ssb}{{\mathbf{s}}}
\newcommand{\tb}{{\mathbf{t}}}
\newcommand{\obb}{{\overline{\mathbf{b}}}}
\newcommand{\ori}{{\mathbf{0}}}

\newcommand{\Oc}{{\mathcal{O}}}
\newcommand{\Fc}{{\mathcal{F}}}
\newcommand{\Cc}{{\mathcal{C}}}
\newcommand{\Dc}{{\mathcal{D}}}
\newcommand{\Zc}{{\mathcal{Z}}}
\newcommand{\Gc}{{\mathcal{G}}}
\newcommand{\Pc}{{\mathcal{P}}}
\newcommand{\Nc}{{\mathcal{N}}}
\newcommand{\Rc}{{\mathcal{R}}}
\newcommand{\Sc}{{\mathcal{S}}}
\newcommand{\Lc}{{\mathcal{L}}}
\newcommand{\Hc}{{\mathcal{H}}}
\begin{document}
\title{On the minimum number of high degree curves containing few points}
\author{Mario Huicochea\\CONACyT/UAZ}
\date{}
\email{dym@cimat.mx}
\begin{abstract}
Let $d,n\in\Nat$ and $A$ be a nonempty finite subset of $\Real^2$. A curve of degree $d$ is the zero set of a polynomial of degree $d$ in $\Real[x,y]$; denote by $\Cc_{d}$ the family of  curves of degree $d$. For any $C_1\in \Cc_{d}$, we say that $C_1$ is determined by $A$ if for any $C_2\in\Cc_{d}$ such that $C_2\cap A\supseteq C_1\cap A$, we have that $C_1=C_2$; we denote by $\Dc_d(A)$ the family of curves of degree $d$ determined by $A$. Write $\Oc_{d,n}(A):=\{C\in \Dc_{d}(A):\;|C\cap A|\leq n\}$. In this paper we state two Sylvester-Gallai type results. In the first one, we show that if there is no $C\in \Cc_d$ containing $A$, then 
\begin{equation*}
\left|\Oc_{d,\frac{3d^2-3d+4}{2}}(A)\right|=\Omega_{d}\left(|A|^{d}\right);
\end{equation*}
moreover we give a construction which shows that this lower bound is the best possible. In the second main result of this paper, it is shown that if $d\geq 3$, there is no $C\in \Cc_d$ containing $A$ and any subset $B$ of $A$ with $|B|=n$ is contained in at most one $C\in \Cc_d$, then 
\begin{equation*}
\left|\Oc_{d,2n+1-\binom{d+2}{2}}(A)\right|=\Omega_{n,d}\left(|A|^{\binom{d+2}{2}-3}\right);
\end{equation*}
furthermore we show that this lower bound is not trivial.
\end{abstract}
\maketitle

\section{Introduction}
In this paper $\Real, \Zet,\Nat, \Natz$ denote the set of real numbers,  integers, positive integers and nonnegative integers, respectively. For any set $X$, we denote by $\Pc(X)$ the family of subsets of $X$ and by $\Pc_n(X)$ the family of subsets $Y$ of $X$ such that $|Y|=n$.  For any $n,m\in\Zet$, we write $[n,m]:=\{k\in\Zet:\;n\leq k\leq m\}$. Let $d,n\in\Nat$.  A \emph{curve of degree $d$} is a subset $C$ of $\Real^2$ which is the zero set of a polynomial in $\Real[x,y]$ of degree $d$; we denote by $\Cc_{d}$ the family of  curves of degree $d$ in $\Real^2$. For each $A\in\Pc(\Real^2)$, we  say that $C_1\in\Cc_{d}$ is \emph{determined by $A$} if for any $C_2\in \Cc_{d}$ satisfying that $C_2\cap A\supseteq C_1\cap A$, we have that $C_1=C_2$; we denote by $\Dc_{d}(A)$ the family of elements of $\Cc_{d}$ which are determined by $A$. We write
\begin{equation*}
\Oc_{d,n}(A):=\{C\in \Dc_{d}(A):\;|C\cap A|\leq n\},
\end{equation*}
and $\Oc_d(A):=\Oc_{d,\binom{d+2}{2}-1}(A)$ . Since there is always a curve of degree $d$ passing through $\binom{d+2}{2}-1$ given points, notice that any $C\in\Dc_d(A)$ satisfies that $|C\cap A|\geq \binom{d+2}{2}-1$, and therefore $\Oc_{d,n}(A)=\emptyset$ for all $n<\binom{d+2}{2}-1$.

One of the best known results in discrete geometry is Sylvester-Gallai theorem and it can be stated as follows. 
\begin{thm}
\label{R1}
Let $A\in\Pc\left(\Real^2\right)$ be finite. If $A$  is not contained in a line, then $\Oc_1(A)\neq\emptyset$. 
\end{thm}
\begin{proof}
See \cite{BM}.
\end{proof}

 Sylvester-Gallai theorem has opened  a complete research field in discrete geometry, see for instance \cite{BM}, \cite{BMP}, \cite{GT}, \cite{PS}. An important Sylvester-Gallai problem is to bound $|\Oc_1(A)|$ in terms of $|A|$. A number of  quantitative results have been found, see \cite{CS}, \cite{GT}, \cite{KM}, \cite{Mo}. One of these results is the Dirac-Motzkin conjecture which was proven by B. Green and T. Tao. 
\begin{thm}
\label{R2}
There is an absolute constant $c_1>0$ with the following property.  Let $A\in\Pc\left(\Real^2\right)$ be finite with $|A|>c_1$. If $A$ is not contained in a line, then
\begin{equation*}
|\Oc_1(A)|\geq \frac{1}{2}|A|.
\end{equation*}
\end{thm}
\begin{proof}
See \cite[Thm. 1.2]{GT}.
\end{proof}
Another important  problem is to find Sylvester-Gallai qualitative type results  for different geometric objects instead of lines (e.g.  conics, circles, hyperplanes, etc.). For conics, this was done  by J. Wiseman and P. Wilson.
\begin{thm}
\label{R3}
Let $A\in\Pc\left(\Real^2\right)$ be finite. If $A$  is not contained in a curve of degree $2$, then $\Oc_2(A)\neq\emptyset$.
\end{thm}
\begin{proof}
See \cite[Thm.]{WW}.
\end{proof}
Different proofs of Theorem \ref{R3} were found later, see \cite{BVZ}, \cite{CDFGLMSST}. However, until this paper, no Sylvester-Gallai type results have been obtained for curves of higher degree. In particular,  the following conjecture of Wiseman and Wilson remains unproven. 
\begin{con}
\label{R4}
Let $d\in\Nat$ and  $A\in\Pc\left(\Real^2\right)$ be finite. If $A$ is not contained in a curve of degree $d$,  then $\Oc_d(A)\neq\emptyset$. 
\end{con}
In the last few years,  a mix of the previous two Sivester-Gallai   type problems have been studied. This means that  given  a nonempty finite subset $A$ in $\Real^2$ and a family $\Fc$  of geometric objects (e.g. circles, conics, hyperplanes, etc.), we want (under reasonable assumptions) to determine or at least bound the number of elements $F$ of $\Fc$ such that $F$ is determined by $A$ and $|F\cap A|$ is small. This problem has been studied at least for lines, circles, conics and hyperplanes, see \cite{Ba}, \cite{BMo}, \cite{BVZ}, \cite{GT}, \cite{Hu2}, \cite{LMNSSZ}, \cite{LS19}, \cite{LS20}. However, there are no results for high  degree curves (as we said above, to our best knowledge, neither qualitative Silvester-Gallai  results exist for high  degree curves). In this paper we are interested in this problem when $\Fc$ is the family of curves with a given degree.  The first result of this paper is the next one.
\begin{thm}
\label{R5}
For each  $d\in\Nat$, there are $c_2=c_2(d),c_3=c_3(d)>0$ with the following property.  Let  $A\in\Pc\left(\Real^2\right)$ be finite with $|A|>c_2$. If $A$ is not contained in a curve of degree $d$,  then
\begin{equation*}
\left|\Oc_{d,\frac{3d^2-3d+4}{2}}(A)\right|\geq c_3|A|^{d}. 
\end{equation*}
\end{thm}
For $d\in\{1,2\}$, note that $\frac{3d^2-3d+4}{2}=\binom{d+2}{2}-1$. Thus Theorem \ref{R5} can be seen as a generalization of  Theorem \ref{R2} (although, in the proof of Theorem \ref{R5}, we do not care about the best possible value of $c_3$), and Theorem \ref{R5} implies Theorem \ref{R3} for big enough sets. Moreover Theorem \ref{R5} is optimal as we see in the next result. 
\begin{thm}
\label{R6}
For each  $d,m\in\Zet$ such that   $m>\max\left\{\frac{3d^2-3d+4}{2}, \frac{d^2+4d}{2}\right\}$ and $d>1$, there is $A\in\Pc_m\left(\Real^2\right)$ such that
\begin{enumerate}
\item[i)]$A$ is not contained in a curve of degree $d$.
\item[ii)]$\left|\Oc_{d,\frac{3d^2-3d+4}{2}}(A)\right|\leq \binom{|A|-\binom{d+1}{2}}{d}$.
\end{enumerate}
\end{thm}

 Unfortunately, for $d>2$, we have that $\frac{3d^2-3d+4}{2}\geq \binom{d+2}{2}$ so Theorem \ref{R5} does not provide information about $\Oc_{d,n}(A)$ when $\binom{d+2}{2}-1\leq n<\frac{3d^2-3d+4}{2}$. The next result of this paper deals with this problem when $A$ satisfies that each $B\in\Pc_n(A)$ is contained in at most one curve of degree $d$.
\begin{thm}
\label{R7}
For each  $d,n\in\Zet$ with $d\geq 3$ and  $n\geq \binom{d+2}{2}-1$, there are $c_4=c_4(n,d),c_5=c_5(n,d)>0$ with the following property.  Let   $A\in\Pc\left(\Real^2\right)$ be finite such that for each  $B\in\Pc_n(A)$, we  have that $B$ is contained in at most one curve of degree $d$. If $|A|>c_4$ and  $A$ is not contained in a curve of degree $d$, then 
\begin{equation*}
\left|\Oc_{d,2n+1-\binom{d+2}{2}}(A)\right|\geq c_5|A|^{\binom{d+2}{2}-3}. 
\end{equation*}
\end{thm}
The values of  $c_2,c_3,c_4$ and $c_5$ can be determined explicitly (although they depend on some values of some auxiliary results). As a consequence of Theorem \ref{R7},  if  any $B\in\Pc_{\binom{d+2}{2}-1}(A)$ is contained in at most one curve of degree $d$, then $\left|\Oc_{d}(A)\right|=\Omega\left(|A|^{\binom{d+2}{2}-3}\right)$ so Theorem \ref{R7} implies a particular case of Conjecture \ref{R4}. Also  Theorem \ref{R7} is not trivial as we see in the next theorem (although we do not know if the lower bound $\Omega\left(|A|^{\binom{d+2}{2}-3}\right)$ is the best one).
\begin{thm}
\label{R8}
For each  $d,n,m\in\Nat$ such that $n\geq \binom{d+2}{2}-1$ and  $m>2n+1-\binom{d+2}{2}$, there is $A\in\Pc_m\left(\Real^2\right)$ such that
\begin{enumerate}
\item[i)]$A$ is not contained in a curve of degree $d$.
\item[ii)]Each  $B\in\Pc_n(A)$ is contained in at most one curve of degree $d$. 
\item[iii)]$\left|\Oc_{d,2n+1-\binom{d+2}{2}}(A)\right|\leq \binom{|A|-1}{\binom{d+2}{2}-2}$.
\end{enumerate}
\end{thm}
 
   We explain the main ideas in the proofs of Theorem \ref{R5} and Theorem \ref{R7}.
 \begin{enumerate}
 \item[i)]Let $d\in\Nat$, $A$ be a nonempty finite subset of $\Real^2$ and $\psi_d$ be the $d$-Veronese map. For each hyperplane $H$ in $\Real^{\binom{d+2}{2}-1}$, we have that $\psi_d^{-1}(H)\in \Cc_{e}$ for some $e\in[1,d]$; conversely, for each $C\in\bigcup_{e=1}^{d}\Cc_e$, there is a hyperplane  $H$ in $\Real^{\binom{d+2}{2}-1}$ such that $\psi_d^{-1}(H)=C$. Thus the problem of finding curves determined by $A$ which contain few points of $A$  is (almost) equivalent to the problem of finding hyperplanes generated  by $\psi_d(A)$ which contain few points of $\psi_d(A)$.
 
 \item[ii)]It is easier to deal with hyperplanes than to do it with curves. As we mentioned above, there are already results that assure the existence of several hyperplanes in $\Real^e$ which are generated but contain few points of a given set $S$, see \cite{Ba}, \cite{BMo}, \cite{LS19}, \cite{LS20}. Nonetheless, these results  require that any $e$ points of $S$ generate a hyperplane. In general, the set $\psi_d(A)$ wont satisfy this condition in $\Real^{\binom{d+2}{2}-1}$. Thus we cannot take advantage of \cite{Ba}, \cite{BMo}, \cite{LS19}, \cite{LS20}.   What we will do is to fix a $\binom{d+2}{2}-4$-dimensional flat $F$ and a $2$-dimensional flat $G$  such that $F\cap G=\emptyset$. Considering $\Real^{\binom{d+2}{2}-1}$ embedded into $\Pro^{\binom{d+2}{2}-1}$ and the  homogenization $G^h\cong \Pro^2$ of $G$ in $\Pro^{\binom{d+2}{2}-1}$ , we will project each $\binom{d+2}{2}-3$-dimensional flat $K$ containing $F$ into $K\cap G^h$ (which is always a single point in the projective space $G^h\cong\Pro^2$); this will induce a map $\pi_F:\Real^{\binom{d+2}{2}-1}\setminus F\rightarrow \Pro^2$. In this way we take the problem of  hyperplanes containing $F$ into a problem of lines. 
 
 \item[iii)]Nevertheless,  we will need very special flats $F$. To be able to say that  if a line $L$ contains few points of $(\pi_{F}\circ\psi_d)(A)$, then the curve $(\pi_{F}\circ\psi_d)^{-1}(L)$ contains few points of $A$, we need that $F$ satisfies certain conditions.   The flats $F$ will be the affine hull of $\psi_d(B)$ where  $B$ is a subset of $A$ which satisfies some technical conditions. The family of subsets $B$ of $A$ satisfying those assumptions will be denoted by $\Nc_d(A)$. The longest and most tedious part of this paper (which is Section 3) is to warranty that  $\Nc_d(A)$ is not very small, however this is the core of this article. The proof of this fact will depend on the structure of $A$; specifically, it depends on whether there is a low degree curve that contains several points of $A$ or not. 
 
 \item[iv)]Let $F$ be the affine hull of $\psi_d(B)$ for some   $B\in \Nc_d(A)$ and denote by  $\varphi $ the restriction of $\pi_{F}\circ\psi_d$ to $\Real^2\setminus\psi_d^{-1}(F)$. In Section 4 we will prove that there are $k\leq d^2$ and  a finite collection (bounded in terms of $d$) of  Zariski closed subsets  $\{E_i\}_{i\in I}$ of $\Real^2$ such that 
 \begin{enumerate}
 \item[$\bullet$] $\varphi(E_i)$ is a singleton for each $i\in I$.
  \item[$\bullet$] For each $\ab\in \Real^2\setminus(\psi_d^{-1}(F)\cup \bigcup_{i\in I}E_i)$, we have that $\varphi(\ab)\not\in\{\varphi(E_i):i\in I\}$.
  \item[$\bullet$] For each $\ab\in \Real^2\setminus(\psi_d^{-1}(F)\cup \bigcup_{i\in I}E_i)$, we have that $|\varphi^{-1}(\varphi(\ab))|\leq k$.
\end{enumerate} 

\item[v)]With the map $\varphi$ as in iv),  any   line $L$ in $\Real^2$ disjoint from $\{\varphi(E_i):i\in I\}$ which contains exactly two points of $\varphi\left(A\setminus \psi_d^{-1}(F)\right)$  satisfies that $\varphi^{-1}(L)$ is an element of $\Cc_{d}$ which contains few points of $A$ and is determined by it. A result of T. Boys , C. Valculescu and  F. de Zeeuw (see Lemma \ref{R19}) will warranty that there are several lines $L$ as above, and therefore we will have a number of the desired curves for each $B\in\Nc_d(A)$. This can be done for each flat $F$ generated by  an element of $\Nc_d(A)$. Thus, since $|\Nc_d(A)|\gg 0$, we will  have several curves satisfying  the desired conditions. 
 \end{enumerate}
 
We think that at least as valuable as the main results of this paper is the method  we use since it allows to translate many Sylvester-Gallai problems of curves into Sylvester-Gallai problems of lines where much more tools are available.

This paper is organized as follows. In Section 2 we establish some notation and auxiliary results that will be needed in the forthcoming sections. As we already mentioned,  the families $\Nc_d(A)$  are fundamental tools in the proofs of the main results. We introduce and state some properties of the families $\Nc_d(A)$ in Section 3. Using the elements of the families $\Nc_d(A)$, we will take the problem of finding curves with few points into finding ordinary lines which avoid a finite set and this will be done in Section 4. The proofs of the main results of this paper are completed in Section 5.  

\section{Preliminaries}
 In this section we state some notation  and  results that will be needed later. 
 
 Let $p(x,y)\in\Real[x,y]$ and $d\in\Nat$. We denote by $\Zc(p(x,y))$ its zero set in $\Real^2$ and by $\deg(p(x,y))$ its degree.   We say that $p(x,y)\in\Real[x,y]$ is \emph{irreducible} if $\deg(p(x,y))>0$ and for any factorization $p(x,y)=p_1(x,y)p_2(x,y)$, we get that $p_i(x,y)\in\Real$ for some $i\in\{1,2\}$. We say that $\Zc(p(x,y))$ is \emph{irreducible} if $p(x,y)$ is irreducible.  Write 
\begin{align*}
\Cc_{\leq d}&:=\bigcup_{e=1}^d\Cc_e\\
\Real_d[x,y]&:=\{p(x,y)\in\Real[x,y]:\;\deg(p(x,y))\in [1,d]\};
\end{align*}
also, for technical reasons, we write $\Cc_{\leq 0}:=\emptyset$.  In $\Real[x,y]$, we define the relation   $p(x,y)\sim q(x,y)$ if there is $r\in\Real\setminus\{0\}$ such that $p(x,y)=r\cdot  q(x,y)$, and we denote by $[p(x,y)]$ the class of $p(x,y)$ and by $\Real[x,y]/\sim$ the set of classes. For any subset $X$ of $\Real[x,y]$, we write
 \begin{equation*}
 X/\sim:=\{[p(x,y)]\in\Real[x,y]/\sim:\;[p(x,y)]\cap X\neq\emptyset\}.
 \end{equation*}
  Note that for any $p(x,y),q(x,y)\in\Real[x,y]$ such that $[p(x,y)]=[q(x,y)]$, we have that $\deg(p(x,y))=\deg(q(x,y))$ and $\Zc(p(x,y))=\Zc(q(x,y))$; thus    
   \begin{equation*}
 \sigma_{d}:\Real_d[x,y]/\sim\longrightarrow \Cc_{\leq d},\qquad \sigma_{d}([p(x,y)])=\Zc(p(x,y)).
 \end{equation*} 
  is well defined. Let $p(x,y)\in\Real[x,y]$ be such that $\deg(p(x,y))>0$ and consider a factorization $p(x,y)=r\prod_{i=1}^np_i(x,y)^{m_i}$ with $m_1,m_2,\ldots, m_n\in\Nat$, $r\in\Real$, $[p_i(x,y)]\neq [p_j(x,y)]$ for all $i,j\in[1,n]$ such that $i\neq j$, and $p_i(x,y)$ is irreducible for each $i\in[1,n]$. Then the irreducible  curves $\Zc(p_1(x,y)), \Zc(p_2(x,y)),\ldots,\Zc(p_n(x,y))$ are known as the \emph{irreducible components of $\Zc(p(x,y))$}. The irreducible components satisfy that
\begin{equation*}
 \Zc(p(x,y))=\bigcup_{i=1}^n\Zc(p_i(x,y))=\Zc\left(\prod_{i=1}^n p_i(x,y)\right)
\end{equation*} 
 and 
\begin{equation*}
\deg(p(x,y))=\sum_{i=1}^nm_i\deg (p_i(x,y))\geq \sum_{i=1}^n\deg(p_i(x,y)).
\end{equation*}
 We will use a weak version of Bezout's theorem.
\begin{thm}
 \label{R9}
 Let $d,e\in\Nat$, $C_1\in\Cc_{\leq d}$  and $C_2\in\Cc_{\leq e}$. If $C_1$ and $C_2$ do not have an irreducible component  in common, then  
 \begin{equation*}
 |C_1\cap C_2|\leq de.
 \end{equation*}
 \end{thm}
 \begin{proof}
 See \cite[Ch. I.7]{Ha}.
 \end{proof}

The next  facts can be proven easily by the reader.
\begin{rem}
\label{R10}
Let $d\in\Nat$ and $e,f\in[1,d]$.
\begin{enumerate}
\item[i)]For any $C_1\in \Cc_{\leq e}$ and $C_2\in \Cc_{\leq f}$, notice that $C_1\cup C_2\in\Cc_{\leq e+f}$.
\item[ii)]For any $C_0\in \Cc_e$ and $C\in\Cc_{\leq d}$ such that $C_0\subseteq C$, there is $C_1\in\Cc_{\leq d-e}$ such that $C=C_0\cup C_1$.
\item[iii)]For any $A\in \Pc\left(\Real^2\right)$ such that $|A|\leq \binom{d+2}{2}-1$, there is  $C\in\Cc_{\leq d}$ such that $A\subseteq C$.
\end{enumerate}
\end{rem}

For a curve $C$, it may happen that there exist $p(x,y),q(x,y)\in\Real[x,y]$ such that $[p(x,y)]\neq [q(x,y)]$ but $\Zc(p(x,y))=C=\Zc(q(x,y))$. The next lemma shows that for each $C\in\Cc_{\leq d}$ there are at most $d^d$ classes $[p(x,y)]\in\Real_d[x,y]/\sim$ such that $C=\Zc(p(x,y))$.
\begin{lem}
\label{R11}
Let  $d\in\Nat$ and $C\in\Cc_{\leq d}$ with pairwise distinct  irreducible components $\Zc(p_1(x,y)), \Zc(p_2(x,y)),\ldots,\Zc(p_n(x,y))$. Then
\begin{equation*}
\sigma^{-1}_{d}(C)=\left\{\left[\prod_{i=1}^np_i^{m_i}\right]\in \Real_d[x,y]/\sim:\;m_1,m_2,\ldots, m_n\in\Nat,\,\sum_{i=1}^nm_i\deg(p_i)\leq d\right\}.
\end{equation*}
Since the number of solutions $(m_1,m_2,\ldots, m_n)\in\Nat^n$ of $\sum_{i=1}^nm_i\deg(p_i)\leq d$ is bounded by $d^d$, we get in particular that
\begin{equation*}
|\sigma^{-1}_{d}(C)|\leq d^d.
\end{equation*}
\end{lem}
\begin{proof}
See \cite[Cor.7]{Hu1}.
\end{proof}

Let $d\in\Natz$ and  $e\in [0,d]$.  A translation $F$  of a vectorial subspace $V$ of $\Real^d$ will be called  a \emph{flat}. We write  $\dim F:=\dim V$, and also if $V$ is an $e$-dimensional subspace,  we say that $F$ is an $e$-flat; in particular, $1$-flats are lines and $d-1$-flats are hyperplanes. The family of $e$-flats in $\Real^d$ will be denoted by $\Gc_{e}$. For any subset $S$ of $\Real^d$,  we denote by $\Fl(S)$ the smallest flat (with respect to $\subseteq$) which contains $S$ and we write $\dim S:=\dim \Fl(S)$. If $S=\emptyset$, we consider $\Fl(S)=\emptyset$ and $\dim S=-1$. If $S=\{\ssb_1,\ssb_2,\ldots, \ssb_n\}$, we write $\Fl(\ssb_1,\ssb_2,\ldots,\ssb_n):=\Fl(S)$. The family of $e$-flats $F$ in $\Real^d$ such that there is a subset $R$ of $S$ satisfying that $F=\Fl(R)$ will be denoted by $\Gc_{e}(S)$.

 A fundamental tool in this paper is the Veronese map.  Write $I_d:=\\\left\{(n,m)\in\Natz^2:\;n+m\in [1,d]\right\}$ so $|I_d|=\binom{d+2}{2}-1$. The \emph{$d$-Veronese map} is the map
 \begin{equation*}
 \psi_{d}:\Real^2\longrightarrow  \Real^{\binom{d+2}{2}-1},\qquad \psi_{d}(a_1,a_2)=\left(a_1^na_2^m\right)_{(n,m)\in I_d}
\end{equation*} 
To avoid confusion, the ring of polynomials which corresponds to $\Real^2$ will be denoted by $\Real[x,y]$ and the ring of polynomials which corresponds to $\Real^{\binom{d+2}{2}-1}$ will be denoted by $\Real[z_{(n,m)}]_{(n,m)\in I_d}$. There is a quite important relation between elements of $\Real_d[x,y]/\sim$ and hyperplanes in $\Real^{\binom{d+2}{2}-1}$ given by the next map 
\begin{align*}
\tau_{d}:\Real_d[x,y]/\sim\longrightarrow \Gc_{\binom{d+2}{2}-2},&\\
 \tau_{d}\left(\left[r_{(0,0)}+\sum_{(n,m)\in I_d}r_{(n,m)}x^ny^m\right]\right)=\Zc\left(r_{(0,0)}+\sum_{(n,m)\in I_d}r_{(n,m)}z_{(n,m)}\right)&
\end{align*}
 The Veronese map has some well-known properties that we will need later. The proof of the following facts can be found in standard algebraic geometry books, see for instance \cite[Ch. I]{Ha}, \cite[Ch. 1]{Sh}.
\begin{rem}
\label{R12}
Let $d\in\Nat$.  
\begin{enumerate}
\item[i)]The map $\psi_{d}$ is an isomorphism onto its image. 
\item[ii)]The map $\tau_d$ is a bijection. Note that  for any $[p(x,y)]\in \Real_d[x,y]/\sim$, we have that 
\begin{equation*}
\psi_{d}(\Zc(p(x,y)))=\tau_{d}([p(x,y)])\cap \psi_{d}(\Real^2).
\end{equation*}
\item[iii)]For all $e\in[1,d]$ and $A\in \Pc\left(\Real^2\right)$, we have that 
\begin{equation*}
\dim\psi_e(A)\leq \dim\psi_d(A).
\end{equation*}
\end{enumerate}
\end{rem} 
A crucial part of this paper is to take the problem of finding curves determined by $A$ containing few points of $A$ into the problem of finding hyperplanes in $\Real^{\binom{d+2}{2}-1}$ generated by $\psi_d(A)$ which contain  few points of $\psi_d(A)$. The main tool to do this is the following lemma. 
\begin{lem}
\label{R13}
Let $d\in \Nat$ and $A\in\Pc\left(\Real^2\right)$ be such that there is no  element of $\Cc_{\leq d}$ which contains $A$. Then 
\begin{equation*}
\tau_{d}\left(\sigma_{d}^{-1}(\Dc_{d}(A))\right)=\Gc_{\binom{d+2}{2}-2}(\psi_{d}(A)).
\end{equation*}
\end{lem}
\begin{proof}
See \cite[Lemma 12]{Hu1}.
\end{proof} 
Let $d\in\Nat$, $p(x,y)\in\Real_d[x,y]$ and $A\in\Pc\left(\Real^2\right)$ be finite. Since $A$ is finite, there is always a line $L=\Zc(rx+sy-t)$ in $\Real^2$ such that $L\cap A=\emptyset$. Therefore 
\begin{align*}
\Zc(p(x,y))\cap A&=\left( \Zc(p(x,y))\cap A\right)\cup (L\cap A)\\
&= \Zc\left(p(x,y)\cdot (rx+sy-t)^{d-\deg(p(x,y))}\right)\cap A.
\end{align*} 
From this observation, we get the next fact. 
\begin{rem}
\label{R14}
Let $d\in\Nat$, $p(x,y)\in\Real_d[x,y]$ and $A\in\Pc\left(\Real^2\right)$ be finite. Then there is $q(x,y)\in\Real[x,y]$ such that $\deg(q(x,y))=d$ and 
\begin{equation*}
\Zc(p(x,y))\cap A=\Zc(q(x,y))\cap A.
\end{equation*}
\end{rem}
\begin{lem}
\label{R15}
Let $e\in\Nat$, $f\in\Natz$,  $F$ be a proper flat in $\Real^{\binom{e+2}{2}-1}$ and $A\in\Pc_{\binom{f+2}{2}}\left(\Real^2\right)$ be such that  $A$ is not contained in an element of $\Cc_{\leq f}$.
\begin{enumerate}
\item[i)]If $e\geq f$, then $|\psi_e(A)\cap F|\leq 1+\dim F$.
\item[ii)]If $e<f$, then $|\psi_e(A)\cap F|\leq \binom{f+2}{2}-\binom{f-e+2}{2}-\binom{e+2}{2}+2+\dim F$.
\item[iii)]If $e<f$, then $\dim\psi_d(A\setminus\psi_e^{-1}(F))\geq \binom{f-e+2}{2}-1$ for all $d\geq f-e$.
\end{enumerate}
  \end{lem}
  \begin{proof}
  First we show i). If $f=0$, then $|A|=1$ so 
  \begin{equation*}
  |\psi_e(A)\cap F|\leq|\psi_e(A)|=1\leq 1+\dim F.
\end{equation*}
Thus, from now on, we assume that $f>0$.  Since $A$ is not contained in an element of $\Cc_{\leq f}$, Remark \ref{R12}.ii implies that $\psi_f(A)$ cannot be contained in a hyperplane of $\Real^{\binom{f+2}{2}-1}$ and therefore
  \begin{equation}
  \label{E1}
  \binom{f+2}{2}-1=\dim \psi_f(A).
  \end{equation}
 In so far as $e\geq f$,  Remark \ref{R12}.iii implies that $ \dim \psi_f(A)\leq \dim \psi_e(A)$ and thereby
   \begin{equation}
   \label{E2}
   \dim \psi_f(A)\leq \dim \psi_e(A)\leq |\psi_e(A)|-1=|A|-1.
  \end{equation}
 Since $|A|= \binom{f+2}{2}$, we get from (\ref{E1}) and (\ref{E2}) that 
  \begin{equation}
  \label{E3}
  \dim \psi_e(A)=|\psi_e(A)|-1=|A|-1.
  \end{equation}
  As a consequence of (\ref{E3}), for any $S\in \Pc(\psi_e(A))$, we get that $\dim S=|S|-1$; in particular, 
  \begin{equation*}
   \dim \psi_e(A)\cap F=|\psi_e(A)\cap F|-1, 
  \end{equation*}
   and hence 
    \begin{equation*}
  |\psi_e(A)\cap F|=1+ \dim \psi_e(A)\cap F\leq 1+\dim F,
  \end{equation*}
  which completes the proof of i).
  
  Now we prove ii). For any flat $E$ in $\Real^{\binom{e+2}{2}-1}$, write $\alpha(E):=\binom{e+2}{2}-2-\dim E$. In so far as $F$ is a proper flat, note that $\alpha(F)\geq 0$. The proof of ii) will be done by induction on $\alpha(F)$. First suppose that $\alpha(F)=0$ so $F$ is a hyperplane in $\Real^{\binom{e+2}{2}-1}$. Then Remark \ref{R12}.ii implies that $\psi_e^{-1}(F)\in\Cc_{\leq e}$. Trivially,
  \begin{equation}
  \label{E4}
  \psi_e^{-1}(F)\cap A\subseteq \psi_e^{-1}(F).
  \end{equation}
  We claim that 
  \begin{equation}
  \label{E5}
  |\psi_e^{-1}(F)\cap A|\leq \binom{f+2}{2}-\binom{f-e+2}{2}.
  \end{equation}
   Indeed, if (\ref{E5}) is false, then 
    \begin{equation*}
   |A\setminus \psi_e^{-1}(F)|=|A|- |\psi_e^{-1}(F)\cap A|<|A|- \binom{f+2}{2}+\binom{f-e+2}{2}=\binom{f-e+2}{2},
  \end{equation*}
   and therefore Remark \ref{R10}.iii implies that there exists $C\in\Cc_{\leq f-e}$ such that 
   \begin{equation}
   \label{E6}
   A\setminus \psi_e^{-1}(F)\subseteq C.
   \end{equation}
   However, from (\ref{E4}) and (\ref{E6}), 
   \begin{equation*}
   A=(\psi_e^{-1}(F)\cap A)\cup (A\setminus \psi_e^{-1}(F))\subseteq \psi_e^{-1}(F)\cup C
   \end{equation*}
   with $\psi_e^{-1}(F)\cup C\in \Cc_{\leq f}$  by  Remark \ref{R10}.i; this contradicts the assumption so (\ref{E5}) is  true. Therefore
   \begin{align*}
   |F\cap \psi_e(A)|&=|\psi_e^{-1}(F)\cap A|\\
    &\leq \binom{f+2}{2}-\binom{f-e+2}{2}&\Big(\text{by } (\ref{E5})\Big)\\
     &= \binom{f+2}{2}-\binom{f-e+2}{2}-\alpha(F)&\Big(\text{since } \alpha(F)=0\Big),
   \end{align*}
    and the basis of induction is complete. Now assume that the claim holds for all flats $E$ such that $0\leq \alpha(E)<\alpha(F)$. Since $A$ is not contained in an element of $\Cc_{\leq f}$ and $e<f$, we get that  $A$ is not contained in an element of $\Cc_{\leq e}$. Thus   Remark \ref{R12}.ii implies that $\psi_e(A)$ is not contained in a hyperplane and therefore $\dim \psi_e(A)=\binom{e+2}{2}-1$. In particular, this means that $\psi_e(A)\setminus F\neq \emptyset$, and then we can fix $\ab\in A\setminus \psi_e^{-1}(F)$. Set $E:=\Fl(F\cup \{\psi_e(\ab)\})$. Since $\psi_e(\ab)\not\in F$, we have that $\dim E=1+\dim F$, and therefore
    \begin{equation}
    \label{E7}
    \alpha(E)+1=\alpha(F).
\end{equation}
In so far as  $\psi_e(\ab)\in \psi_e(A)\setminus F$, we get that
\begin{equation}
\label{E8}
|F\cap \psi_e(A)|+1\leq |E\cap \psi_e(A)|.
\end{equation} 
Then
\begin{align*}
|F\cap \psi_e(A)|&\leq |E\cap \psi_e(A)|-1&\Big(\text{by } (\ref{E8})\Big)\\
 &\leq  \binom{f+2}{2}-\binom{f-e+2}{2}-\alpha(E)-1&\Big(\text{by induction}\Big)\\
 &=  \binom{f+2}{2}-\binom{f-e+2}{2}-\alpha(F),&\Big(\text{by } (\ref{E7})\Big)
\end{align*} 
 and this completes the induction and the proof of ii).
 
 Finally, we show iii). Since $F$ is a proper flat, there is a hyperplane $H$ which contains $F$.   Remark \ref{R12}.ii implies that $\psi_e^{-1}(H)\in\Cc_{\leq e}$. Since $H\supseteq F$,  
  \begin{equation}
  \label{E9}
  \psi_e^{-1}(F)\cap A\subseteq \psi_e^{-1}(H)\cap A\subseteq \psi_e^{-1}(H).
  \end{equation} 
We claim that   
  \begin{equation}
  \label{E10}
  \dim \psi_{f-e}(A\setminus \psi_e^{-1}(F))\geq \binom{f-e+2}{2}-1.
  \end{equation}
   Indeed, if (\ref{E10}) does not hold, then $\dim  \psi_{f-e}(A\setminus \psi_e^{-1}(F))< \binom{f-e+2}{2}-1$, and hence there exists a hyperplane $K$ in $\Real^{\binom{f-e+2}{2}-1}$ which contains $\psi_{f-e}(A\setminus \psi_e^{-1}(F))$ with  $\psi_{f-e}^{-1}(K)\in\Cc_{\leq f-e}$ by Remark \ref{R12}.ii.   Then 
  \begin{equation}
  \label{E11}
  A\setminus \psi_e^{-1}(F)\subseteq \psi_{f-e}^{-1}(K).
  \end{equation} 
   Nonetheless, from (\ref{E9}) and (\ref{E11}), 
   \begin{equation*}
   A=(\psi_e^{-1}(F)\cap A)\cup (A\setminus \psi_e^{-1}(F))\subseteq \psi_e^{-1}(H)\cup \psi_{f-e}^{-1}(K),
   \end{equation*}
   and Remark \ref{R10}.i yields that $\psi_e^{-1}(H)\cup \psi_{f-e}^{-1}(K)\in\Cc_{\leq f}$; this contradicts the assumption so (\ref{E10}) needs to be  true.  Now, in so far as $d\geq f-e$, Remark \ref{R12}.iii yields that   
  \begin{equation}
  \label{E12}
  \dim \psi_{d}(A\setminus \psi_e^{-1}(F))\geq \dim \psi_{f-e}(A\setminus \psi_e^{-1}(F)).
  \end{equation}
  Then iii) is a consequence of (\ref{E10}) and (\ref{E12}).
  \end{proof}
  \begin{lem}
\label{R16}
Let $e\in\Natz$ and  $A\in\Pc\left(\Real^2\right)$ be such that  $A$ is finite and it is not contained in an element of $\Cc_{\leq e}$. Denote by  $\Rc$  the family of elements  $B\in\Pc_{\binom{e+2}{2}}(A)$ such that $B$ is not contained in an element of $\Cc_{\leq e}$. Then
\begin{equation*}
|\Rc|\geq \frac{1}{2^{\binom{e+2}{2}-1}}|A|.
\end{equation*}
  \end{lem}
  \begin{proof}
  Let $d\in\Nat$ and $T\in\Pc\left(\Real^d\right)$ be such that $\dim T=d$. Set $\Sc_d(T):=\\\left\{R\in\Pc_{d+1}\left(T\right):\;\dim R=d\right\}$. First we show that 
  \begin{equation}
  \label{E13}
  |\Sc_d(T)|\geq \frac{1}{2^d}|T|.
  \end{equation}
  We prove (\ref{E13}) by induction on $d$. If $d=1$, then 
  \begin{equation*}
  \Sc_1(T)=\left\{R\in\Pc_{2}\left(T\right):\;\dim R=1\right\}=\Pc_{2}\left(T\right)
  \end{equation*}
   so $|\Sc_1(T)|=\binom{|T|}{2}\geq \frac{1}{2}|T|$, and the basis of induction is complete. Assume that (\ref{E13}) holds for $d-1$ and we show it for $d$. Since $\dim T=d$, there exists $S\in \Pc_{d}(T)$ such that $\dim S=d-1$; fix $S\in \Pc_{d}(T)$ such that $\dim S=d-1$ and write $H:=\Fl(S)$. We have two cases.
   \begin{enumerate}
   \item[$\bullet$]Assume that $|T\cap H|\geq \frac{1}{2}|T|$. Fix $\tb\in T\setminus H$. Notice that for all $R\in \Sc_{d-1}(T\cap H)$, we have that $|R\cup\{\tb\}|=|R|+1=d$ and $\dim R\cup\{\tb\}=1+\dim R=d$ so $R\cup\{\tb\}\in \Sc_d(T)$. This map
   \begin{equation*}
   \Sc_{d-1}(T\cap H)\rightarrow \Sc_{d}(T),\qquad R\mapsto R\cup\{\tb\}
   \end{equation*}
    is injective so 
   \begin{equation}
   \label{E14}
  |\Sc_{d}(T)|\geq |\Sc_{d-1}(T\cap H)|.
\end{equation}    
By induction, $|\Sc_{d-1}(T\cap H)|\geq \frac{1}{2^{d-1}}|T\cap H|$ so (\ref{E14}) leads to
\begin{equation*}
 |\Sc_{d}(T)|\geq |\Sc_{d-1}(T\cap H)|\geq \frac{1}{2^{d-1}}|T\cap H|\geq \frac{1}{2^d}|T|,
\end{equation*}
 and this completes the induction in this case.
  \item[$\bullet$]Assume that $|T\cap H|< \frac{1}{2}|T|$. Note that for all $\tb\in T\setminus H$, we have that $|S\cup\{\tb\}|=|S|+1=d$ and $\dim S\cup\{\tb\}=1+\dim S=d$ so $S\cup\{\tb\}\in \Sc_d(T)$. The map
   \begin{equation*}
   T\setminus H\rightarrow \Sc_{d}(T),\qquad \tb\mapsto S\cup\{\tb\}
   \end{equation*}
    is injective so 
   \begin{equation*}
    |\Sc_{d}(T)|\geq |T\setminus H|=|T|-|T\cap H|>\frac{1}{2}|T|\geq \frac{1}{2^d}|T|,
\end{equation*}
and this completes the induction.   
   \end{enumerate}
   
  If $e=0$, then $\Cc_{\leq e}=\emptyset$ so $\Rc=\{\{\ab\}:\;\ab\in A\}$ which means that $|\Rc|=|A|$. Thus, from now on, assume that $e>0$. Since $A$ is not contained in an element of $\Cc_{\leq e}$, Remark \ref{R12}.ii implies that $\psi_e(A)$ is not contained in a hyperplane and then $\dim \psi_e(A)=\binom{e+2}{2}-1$. Applying (\ref{E13}) to $\psi_e(A)$ and  $\binom{e+2}{2}-1$, we get that
   \begin{equation}
   \label{E15}
   \left|\Sc_{\binom{e+2}{2}-1}(\psi_e(A))\right|\geq \frac{1}{2^{\binom{e+2}{2}-1}}|\psi_e(A)|=\frac{1}{2^{\binom{e+2}{2}-1}}|A|.
   \end{equation}
   For each $R\in \Sc_{\binom{e+2}{2}-1}(\psi_e(A))$, we have that  $\dim R=\binom{e+2}{2}-1$ so $R$ is not contained in a hyperplane, and hence, by Remark \ref{R12}.ii, $\psi_e^{-1}(R)$ is not contained in an element of $\Cc_{\leq e}$. Thus the map
      \begin{equation*}
  \Sc_{\binom{e+2}{2}-1}(\psi_e(A))\rightarrow \Rc,\qquad R\mapsto \psi_e^{-1}(R)
   \end{equation*}
    is well defined and injective yielding that
   \begin{equation*}
     |\Rc|\geq \left|\Sc_{\binom{e+2}{2}-1}(\psi_e(A))\right|,
\end{equation*} 
 and finally (\ref{E15}) implies the claim. 
   \end{proof}
   
   \begin{lem}
  \label{R17}
   Let $d\in\Nat$, $f\in[0,d-1]$,  $C_0\in\Cc_{d-f}$  and $B_0\in\Pc_{\binom{f+2}{2}}(\Real^2\setminus C_0)$ be such that $B_0$ is not contained in an element of $\Cc_{\leq f}$. Take  $B\in \Pc_{\binom{d+2}{2}-3}(\Real^2)$ such that   $B\supseteq B_0$ and  $B\cap C_0=B\setminus B_0$. Then, for any $e\in [d-f,d]$ and  $C\in\Cc_{\leq e}$ such that $C\supseteq C_0$, we have that 
\begin{equation}
\label{E16}
|B\cap C|< \binom{d+2}{2}-\binom{d-e+2}{2}-2.
\end{equation}     
  \end{lem}
  \begin{proof}
  Since $C_0\subseteq C$, Remark \ref{R10}.ii implies that there is $C_1\in\Cc_{\leq e+f-d}$ such that $C=C_0\cup C_1$. In so far as $B_0\cap C_0=\emptyset$, notice that 
\begin{equation}
\label{E17}
B_0\cap C=B_0\cap (C_0\cup C_1)=B_0\cap C_1\subseteq C_1.
\end{equation} 
If (\ref{E16}) is false, then
\begin{equation}
\label{EE17}
|B_0\setminus C|\leq |B\setminus C|=|B|-|B\cap C|\leq \binom{d-e+2}{2}-1.
\end{equation}
If $e=d$, then (\ref{EE17}) implies that $B_0\subseteq C=C_0\cup C_1$; hence, inasmuch as $B_0\cap C_0=\emptyset$, we get that $B_0\subseteq C_1$ but   this contradicts that $B_0$ is not contained in an element of $\Cc_{\leq f}$. If $e<d$, then, by (\ref{EE17}),    Remark \ref{R10}.iii yields the existence of  $C_2\in\Cc_{\leq d-e}$ such that 
 \begin{equation}
 \label{E18}
 B_0\setminus C\subseteq C_2.
 \end{equation}
Since $C_1\in\Cc_{\leq e+f-d}$  and $C_2\in\Cc_{\leq d-e}$, note that $C_1\cup C_2\in \Cc_{\leq f}$ by Remark \ref{R10}.i. Notice that $B_0\subseteq C_1\cup C_2$ by (\ref{E17}) and (\ref{E18}); however, this contradicts that $B_0$ is not contained in an element of $\Cc_{\leq f}$ and it proves (\ref{E16}).
  \end{proof}
  
  \begin{lem}
\label{R18}
Let $d\in\Nat$, $e\in[1,d-1]$,  $C\in\Cc_{e}$  and $B\in\Pc(\Real^2)$.
\begin{enumerate}
\item[i)]If $\dim \psi_{d-e}(B\setminus C)=\binom{d-e+2}{2}-1$, then $\dim \psi_{d}(B\cup C)=\binom{d+2}{2}-1$.
\item[ii)]If $\dim \psi_{d-e}(B\setminus C)=\binom{d-e+2}{2}-2$, then $\dim \psi_{d}(B\cup C)=\binom{d+2}{2}-2$.
\end{enumerate}
\end{lem}
\begin{proof}
First we show i). Assume that $\dim \psi_{d}(B\cup C)<\binom{d+2}{2}-1$ so $\psi_{d}(B\cup C)$ is contained in a hyperplane of $\Real^{\binom{d+2}{2}-1}$. Then Remark \ref{R12}.ii implies that there is $C_1\in\Cc_{\leq d}$ such that $B\cup C\subseteq C_1$; in particular, $C\subseteq C_1$ and then Remark \ref{R10}.ii yields the existence of $C_2\in\Cc_{\leq d-e}$ such that $C_1=C\cup C_2$. In so far as $B\cup C\subseteq C_1=C\cup C_2$, we get that $B\setminus C\subseteq C_2$. Inasmuch as $C_2\in\Cc_{\leq d-e}$,  Remark \ref{R12}.ii  implies that $\psi_{d-e}(B\setminus C)$  is contained in a hyperplane of $\Real^{\binom{d-e+2}{2}-1}$ and therefore $\dim \psi_{d-e}(B\setminus C)<\binom{d-e+2}{2}-1$, which proves i).

Now we prove ii). Assume that $\dim \psi_{d}(B\cup C)\neq \binom{d+2}{2}-2$ and we will show that 
\begin{equation}
\label{E19}
\dim \psi_{d-e}(B\setminus C)\neq\binom{d-e+2}{2}-2.
\end{equation}
 We have to deal with two cases.
\begin{enumerate}
\item[$\bullet$]Suppose that  $\dim \psi_{d}(B\cup C)> \binom{d+2}{2}-2$ so $\dim \psi_{d}(B\cup C)=\binom{d+2}{2}-1$. We claim that 
\begin{equation}
\label{E20}
\dim \psi_{d-e}(B\setminus C)=\binom{d-e+2}{2}-1.
\end{equation}
Indeed, if (\ref{E20}) is false, there is a hyperplane of $\Real^{\binom{d-e+2}{2}-1}$ which contains $\psi_{d-e}(B\setminus C)$. Then, by Remark \ref{R12}.ii, there is a curve $C'\in\Cc_{\leq d-e}$ such that $B\setminus C\subseteq C'$. Hence
\begin{equation*}
B\cup C=(B\setminus C)\cup C\subseteq C'\cup C
\end{equation*}
 with $C'\cup C\in \Cc_{\leq d}$  by Remark \ref{R10}.i. This means that $\psi_d(B\cup C)$ is contained in a hyperplane by Remark \ref{R12}.ii, and thus $\dim \psi_{d}(B\cup C)<\binom{d+2}{2}-1$. This contradiction proves (\ref{E20}).
\item[$\bullet$]Suppose that $\dim \psi_{d}(B\cup C)< \binom{d+2}{2}-2$. On the one hand,  there are infinitely many hyperplanes in $\Real^{\binom{d+2}{2}-1}$ containing $\psi_d(B\cup C)$. On the other hand, for each curve $C'\in\Cc_{\leq d}$, there are only finitely many $[p(x,y)]\in\Real_d[x,y]/\sim$ such that $C'=\Zc(p(x,y))$ by Lemma \ref{R11}. Therefore we can choose $p_1(x,y),p_2(x,y)\in\Real_d[x,y]$ such that $\psi_d(B\cup C)\subseteq \tau_d([p_1(x,y)])$, $\psi_d(B\cup C)\subseteq \tau_d([p_2(x,y)])$ and $\Zc(p_1(x,y))\neq \Zc(p_2(x,y))$; write $C_1:=\Zc(p_1(x,y))$ and $C_2:=\Zc(p_2(x,y))$. Since $\psi_d(B\cup C)\subseteq \tau_d([p_1(x,y)])$, Remark \ref{R12}.ii implies that $B\cup C\subseteq C_1$. Then  Remark \ref{R10}.ii yields the existence of $C_3\in\Cc_{\leq d-e}$ such that $C_1=C\cup C_3$. Since  $B\cup C\subseteq C_1=C\cup C_3$, we get that $B\setminus C\subseteq C_3$. In so far as $C_3\in\Cc_{\leq d-e}$, Remark \ref{R12}.ii implies that there is a hyperplane $H_3$ in $\Real^{\binom{d-e+2}{2}-1}$ such that $\psi_{d-e}^{-1}(H_3)=C_3$ and  $\psi_{d-e}(B\setminus C)\subseteq H_3$. Proceeding in the same way with $C_2$, there exist $C_4\in\Cc_{\leq d-e}$ and a hyperplane $H_4$ in $\Real^{\binom{d-e+2}{2}-1}$ such that $C_2=C\cup C_4$,  $\psi_{d-e}^{-1}(H_4)=C_4$ and  $\psi_{d-e}(B\setminus C)\subseteq H_4$. Since $C_1\neq C_2$, notice that $C_3\neq C_4$ and therefore $H_3\neq H_4$. Now, since 
\begin{equation*}
\psi_{d-e}(B\setminus C)\subseteq H_3\cap H_4, 
\end{equation*} 
we get that  $\dim \psi_{d-e}(B\setminus C)<\binom{d-e+2}{2}-2$.
\end{enumerate}
Therefore, in any case, (\ref{E19}) is true and this proves ii). 
\end{proof}
\begin{lem}
\label{R19}
For any $n\in\Nat$, there are $c_6=c_6(n),c_7=c_7(n)\in\Real$ with the following property. For all $T\in\Pc_n\left(\Real^2\right)$ and $S\in\Pc\left(\Real^2\right)$ such that $|S|>c_6$ and $S\setminus T$ is not collinear, we have that 
\begin{equation*}
|\{L\in\Oc_2(S):\;L\cap T=\emptyset\}|\geq \frac{1}{2}|S|-c_7.
\end{equation*}
\end{lem}
\begin{proof}
See \cite[Lemma 2.5]{BVZ}.
\end{proof}
 \section{Families $\Nc_d(A)$}
   The purpose of this section is to define and prove some properties of the families $\Nc_d(A)$ which will be fundamental objects in the proof of the main results of this paper. 
   
   Let $d\in\Nat$ and $A\in\Pc(\Real^2)$. We denote by $\Nc_d(A)$ the family of subsets $B\in\Pc_{\binom{d+2}{2}-3}(A)$ which have following four properties:
   \begin{enumerate}
   \item[i)]$\dim \psi_d(B)=\binom{d+2}{2}-4$.
     \item[ii)]For all $e\in[1,d-1]$ and $C\in\Cc_e$, we have that $|B\cap C|<\binom{d+2}{2}-\binom{d-e+2}{2}$.
      \item[iii)]For all $e\in[1,d-1]$ and $C\in\Cc_e$ such that  $|B\cap C|=\binom{d+2}{2}-\binom{d-e+2}{2}-1$, we have that $\dim\psi_{d-e}(B\setminus C)=\binom{d-e+2}{2}-3$.
      \item[iv)]For all $e\in[1,d-1]$ and $C\in\Cc_e$ such that  $|B\cap C|<\binom{d+2}{2}-\binom{d-e+2}{2}-1$, we have that $\dim\psi_{d-e}(B\setminus C)>\binom{d-e+2}{2}-3$.
   \end{enumerate}
   
   We will need more notation.  For $d\in\Nat$, $e\in[1,d-1]$, $A,B,C \in\Pc(\Real^2)$ and $D\in\Pc(B)$, write
   
   \begin{align*}
   V_e(D)&:=\Fl(\psi_e(D))\\
   W_e(B,D)&:=\Fl(\psi_{d-e}(B\setminus \psi^{-1}_e(V_e(D))))\\
   \alpha_e(D)&:=\binom{e+2}{2}-2-\dim V_e(D)\\
   \beta_e(B,D)&:=\binom{d-e+2}{2}-3-\dim W_e(B,D)\\
   \gamma_e(B,D)&:=|V_e(D)\cap \psi_e(B)|\\
   \mu_e(B,D)&:=\left\{ \begin{array}{ll}
0& \text{if $\alpha_e(D)<0$};\\
&\\
 \alpha_e(D)+ \gamma_e(B,D)+\binom{d-e+2}{2} & \text{if $\alpha_e(D)\geq 0$}.\end{array} \right.
   \end{align*}
   \begin{align*}
\tau_e(B,D)&:=\left\{ \begin{array}{ll}
0& \text{if $\min\{\alpha_e(D),\beta_e(B,D)\}<0$ or}\\
&\text{$\gamma_e(B,D)>\binom{d+2}{2}-\binom{d-e+2}{2}-1$};\\
&\\
 \alpha_e(D)+ \beta_e(B,D)+|B|+2 & \text{if $\min\{\alpha_e(D),\beta_e(B,D)\}\geq 0$ and}\\
&\text{$\gamma_e(B,D)=\binom{d+2}{2}-\binom{d-e+2}{2}-1$};\\
&\\
 \alpha_e(D)+ \beta_e(B,D)+|B|+3 & \text{if $\min\{\alpha_e(D),\beta_e(B,D)\}\geq 0$ and}\\
&\text{$\gamma_e(B,D)<\binom{d+2}{2}-\binom{d-e+2}{2}-1$}.\end{array} \right.
   \end{align*}
     \begin{align*}
U_e(B,D)&:=\left\{ \begin{array}{ll}
\psi_d^{-1}(V_d(B))& \text{if $\alpha_e(D)<0$};\\
&\\
 \psi_d^{-1}(V_d(B))\cup\psi_e^{-1}(V_e(D)) & \text{if $\beta_e(B,D)<0\leq\alpha_e(D)$};\\
&\\
 \psi_d^{-1}(V_d(B))\cup\psi_e^{-1}(V_e(D))\cup \psi_{d-e}^{-1}(W_e(B,D)) & \text{if $\beta_e(B,D),\alpha_e(D)\geq 0$}.\end{array} \right.
   \end{align*}
   \begin{align*}
   I(B,C)&:=\{(f,E)\in [1,d-1]\times \Pc(B):\;C\nsubseteq U_f(B,E)\}\\
   \Nc_d(A,B,C)&:=\{E\in\Nc_d(A):\;B\in\Pc(E)\text{ and }E\cap C=E\setminus B\}.
   \end{align*}
   The following facts are direct consequences of the definitions.
   \begin{rem}
   \label{R20}
   Let $d\in\Nat$, $e\in[1,d-1]$, $B\in \Pc\left(\Real^2\right)$ and $D\in\Pc(B)$.
   \begin{enumerate}
   \item[i)]Note that 
   \begin{equation*}
   B\subseteq \left(B\cap \psi_e^{-1}(V_e(D))\right)\cup \left(B\cap \psi_{d-e}^{-1}(W_e(B,D))\right);
   \end{equation*}
    in particular, 
     \begin{equation*}
   |B|\leq \left|B\cap \psi_e^{-1}(V_e(D))\right|+ \left|B\cap \psi_{d-e}^{-1}(W_e(B,D))\right|.
   \end{equation*}
   \item[ii)]Let  $B_2\in \Pc\left(\Real^2\right)$, $B_1,D_2\in\Pc(B_2)$ and $D_1\in\Pc(B_1)$.  If  $V_e(D_1)\subseteq V_e(D_2)\neq \Real^{\binom{e+2}{2}-1}$ and $W_e(B_1,D_1)=W_e(B_2,D_2)$, then 
   \begin{equation*}
   U_e(B_1,D_1)\subseteq U_e(B_2,D_2). 
   \end{equation*}
    \item[iii)]Let $B_2\in \Pc\left(\Real^2\right)$, $B_1,D_2\in\Pc(B_2)$ and $D_1\in\Pc(B_1)$.  If  $V_e(D_1)= V_e(D_2)$, then $W_e(B_1,D_1)\subseteq W_e(B_2,D_2)$ and
   \begin{equation*}
   U_e(B_1,D_1)\subseteq U_e(B_2,D_2). 
   \end{equation*}
        \item[iv)]If $\min\{\alpha_e(D),\beta_e(B,D)\}\geq 0$, then 
     \begin{equation*}
     \tau_e(B,D)\leq  \alpha_e(D)+ \beta_e(B,D)+|B|+3.
     \end{equation*}
\end{enumerate}    
   \end{rem}
    The main results of this section are Lemma \ref{R24} and Lemma \ref{R28}. Their proofs are rather technical so, for the sake of comprehension, we sketch them before we state the auxiliary results that we need to prove them. We will say that $A\in\Pc(\Real^2)$ is \emph{$d$-regular} if $|A\cap C|\leq \frac{1}{2^{2^{3d+8}}}|A|$ for all $C\in\Cc_{\leq d}$. 
    \begin{enumerate}
    \item[$\bullet$]We start with Lemma \ref{R24}.  Assume that $A$ is  $d$-regular and  take  $B_0:=\emptyset$. Given $\obb=(\bb_1,\bb_2,\ldots, \bb_{\binom{d+2}{2}-3})\in A^{\binom{d+2}{2}-3}$, we will construct recursively a chain $B_0\subsetneq B_1\subsetneq \ldots\subsetneq B_{\binom{d+2}{2}-3}$ with $B_{i+1}=B_i\cup\{\bb_{i+1}\}$ for each $i\in[0,\binom{d+2}{2}-4]$. The idea is that for many $\obb\in A^{\binom{d+2}{2}-3}$, we get $B_{\binom{d+2}{2}-3}\in \Nc_d(A,\emptyset,\Real^2)$. On the one hand, for each $i\in[0,\binom{d+2}{2}-4]$, if $\bb_{i+1}$ is not   in a forbidden subset $U_i$ of $\Real^2$, then  $B_{\binom{d+2}{2}-3}\in \Nc_d(A,\emptyset,\Real^2)$; here we will use Lemma \ref{R22} and Lemma \ref{R23}. On the other hand, since $A$ is $d$-regular,  for each $i\in[0,\binom{d+2}{2}-4]$,  the intersection of the forbidden subset $U_i$ and $A$ is very small so $B_{\binom{d+2}{2}-3}\in \Nc_d(A,\emptyset,\Real^2)$ almost always.
    
    \item[$\bullet$]Now we sketch the proof  of Lemma \ref{R28}.   Given $\obb=\\(\bb_{\binom{f+2}{2}+1},\bb_{\binom{f+2}{2}+2},\ldots, \bb_{\binom{d+2}{2}-3})\in (A\cap C_0)^{\binom{d+2}{2}-\binom{f+2}{2}-3}$, we will construct recursively a chain $B_0\subsetneq B_1\subsetneq \ldots\subsetneq B_{\binom{d+2}{2}-\binom{f+2}{2}-3}$ with $B_{i+1}=B_i\cup\{\bb_{\binom{f+2}{2}+i+1}\}$ for each $i\in[0,\binom{d+2}{2}-\binom{f+2}{2}-4]$. As in Lemma \ref{R24}, for many $\obb\in(A\cap C_0)^{\binom{d+2}{2}-\binom{f+2}{2}-3}$, we get $B_{\binom{d+2}{2}-3}\in \Nc_d(A,B_0,C_0)$. As in the sketch of Lemma \ref{R24}, we need that $\bb_{\binom{f+2}{2}+i+1}$  is not in a forbidden subset $U_i$ of $\Real^2$ for each $i\in[0,\binom{d+2}{2}-\binom{f+2}{2}-4]$; most of auxiliary results of this section are used to show this.  Now, in this case, we have that  $A\cap U_i$ is very small because of Bezout's theorem (applied to some curves that intersect $C_0$). 
    \end{enumerate}

   The next properties will shorten the proofs of some auxiliary results
   \begin{lem}
      \label{R21}
      Let $e,d\in \Nat$ with $e\in[1,d-1]$, $B\in \Pc\left(\Real^2\right)$ and $D\in\Pc(B)$.
\begin{enumerate}
\item[i)]If $\alpha_e(D)\geq 0$, then  there is $C\in\Cc_{\leq e}$ such that $\psi_e^{-1}(V_e(D))\subseteq C$
\item[ii)]If $\beta_e(B,D)\geq -1$, then  there is $C\in\Cc_{\leq d-e}$ such that $\psi_{d-e}^{-1}(W_{d-e}(D))\subseteq C$
\item[iii)] If $|B|\leq \binom{d+2}{2}-1$,  then  there is $C\in\Cc_{\leq d}$ such that $\psi_{d}^{-1}(V_d(B))\subseteq C$.
\item[iv)] If $|B|\leq \binom{d+2}{2}-1$,  then there are $C_1\in\Cc_{\leq e},\,C_2\in\Cc_{\leq d-e} $ and $C_3\in\Cc_{\leq d}$ such that $U_e(B,D)\subseteq C_1\cup C_2\cup C_3$. 
\end{enumerate}      
    \end{lem}  
     \begin{proof}
       First we show i). In so far as $\alpha_e(D)\geq 0$, notice that $V_e(D)$ is a proper flat in $\Real^{\binom{e+3}{2}-1}$. Hence there exists a hyperplane $H$ containing $V_e(D)$. Remark \ref{R12}.ii implies that $\psi_e^{-1}(H)$ is in $\Cc_{\leq e}$ and it satisfies that  $\psi^{-1}_e(V_e(D))\subseteq \psi_e^{-1}(H)$, and this proves i).

We prove ii). Since  $\beta_e(B,D)\geq -1$,   $W_e(B,D)$ is a proper flat in $\Real^{\binom{d-e+2}{2}-1}$.  This implies that there exists a hyperplane $H$ containing $W_e(B,D)$. Remark \ref{R12}.ii yields that $\psi_{d-e}^{-1}(H)$ is in $\Cc_{\leq d-e}$ and it satisfies that $ \psi^{-1}_{d-e}(W_e(B,D))\subseteq \psi_{d-e}^{-1}(H)$, which proves ii).

 To prove iii),  note that since  $|B|\leq \binom{d+2}{2}-1$,  
      \begin{equation*}
         \dim V_d(B)=\dim \Fl(\psi_d(B))\leq |\psi_d(B)|-1=|B|-1\leq \binom{d+2}{2}-2.
      \end{equation*}
     Thus there is a hyperplane $H$ containing $V_d(B)$, and then, by Remark \ref{R12}.ii, $\psi_d^{-1}(H)$ is in $\Cc_{\leq d}$ and it satisfies that $\psi^{-1}_d(V_d(B))\subseteq \psi_d^{-1}(H)$.

    Finally, notice that iv) is a consequence of i), ii) and iii).
       \end{proof} 
       \begin{lem}
   \label{R22}
   Let $d\in\Nat$, $A,C\in\Pc\left(\Real^2\right)$ and $B\in\Pc(A)$, and write $U:=\bigcup_{(e,D)\in I(B,C)}U_e(B,D)$. Assume that $A\setminus U\neq\emptyset$ and $\max\left\{\tau_e(B,D),\mu_e(B,D)\right\}<\binom{d+2}{2}$ for all $(e,D)\in I(B,C)$. Then, for all $\bb\in A\setminus U$ and $(e,D)\in  I(B\cup\{\bb\},C)$,
   \begin{equation*}
   \max\left\{\tau_e(B\cup\{\bb\},D),\mu_e(B\cup\{\bb\},D)\right\}<\binom{d+2}{2}.
   \end{equation*}
   \end{lem}
   \begin{proof}
   Fix $\bb\in A\setminus U$  and $(e,D)\in  I(B\cup\{\bb\},C)$. If $\alpha_e(D)<0$, then $\tau_e(B\cup\{\bb\},D)$ and $\mu_e(B\cup\{\bb\},D)$ are zero so we assume from now on that
   \begin{equation}
   \label{E21}
   \alpha_e(D)\geq 0.
\end{equation} 
We divide the proof of the claim into two cases.
\begin{enumerate}
\item[$\bullet$] First assume that $\bb\not\in D$. Then $D$ is a subset of $B$ and Remark \ref{R20}.iii implies that $U_e(B,D)\subseteq U_e(B\cup\{\bb\},D)$; this implies that $(e,D)\in  I(B,C)$ since $(e,D)\in  I(B\cup\{\bb\},C)$. The inclusion  $(e,D)\in  I(B,C)$ leads to
\begin{equation}
\label{E22}
\max\left\{\tau_e(B,D),\mu_e(B,D)\right\}<\binom{d+2}{2}.
\end{equation}
Because $(e,D)\in  I(B,C)$, we have that  $\bb\not\in U_e(B,D)$. Notice that  (\ref{E21}) yields that $\psi^{-1}_e(V_e(D))\subseteq U_e(B,D)$ so $\bb\not\in \psi^{-1}_e(V_e(D))$ and thereby
\begin{align}
\label{E23}
\gamma_e(B\cup\{\bb\},D)=|V_e(D)\cap \psi_e(B\cup\{\bb\})|=|V_e(D)\cap \psi_e(B)|=\gamma_e(B,D).
\end{align} 
Thus 
\begin{align}
\label{E24}
\mu_e(B\cup\{\bb\},D)&=\alpha_e(D)+ \gamma_e(B\cup\{\bb\},D)+\binom{d-e+2}{2}&\Big(\text{by (\ref{E21})}\Big)\nonumber\\
&=\alpha_e(D)+ \gamma_e(B,D)+\binom{d-e+2}{2}&\Big(\text{by (\ref{E23})}\Big)\nonumber\\
&=\mu_e(B,D)&\Big(\text{by (\ref{E21})}\Big)\nonumber\\
&<\binom{d+2}{2}.&\Big(\text{by (\ref{E22})}\Big)
\end{align}
If $\beta_e(B,D)<0$, then $\dim W_e(B,D)>\binom{d-e+2}{2}-3$ so 
\begin{equation*}
\dim W_e(B\cup\{\bb\},D)\geq \dim W_e(B,D)>\binom{d-e+2}{2}-3,
\end{equation*}
 and therefore $\beta_e(B\cup\{\bb\},D)<0$; thus, if  $\beta_e(B,D)<0$, we have that $\tau_e(B\cup\{\bb\},D)=0$, and then by (\ref{E24}) we are done. From now on, we assume that 
  \begin{equation}
   \label{E25}
   \beta_e(B,D)\geq 0.
\end{equation}
From (\ref{E21}) and (\ref{E25}), note that $\alpha_e(D),\beta_e(B,D)\geq 0$ so, since  $\bb\not\in U_e(B,D)$, notice that $\bb\not\in\psi^{-1}_{e}(V_e(D))$ and $\bb\not\in\psi^{-1}_{d-e}(W_e(B,D))$; this leads to $\psi_{d-e}(\bb)\\\in W_e(B\cup\{\bb\},D)\setminus W_e(B,D)$, and then $W_e(B,D)$ is a flat properly contained in $W_e(B\cup\{\bb\},D)$ implying that 
\begin{equation*}
\dim W_e(B\cup\{\bb\},D)\geq 1+\dim W_e(B,D),
\end{equation*}
 and therefore
\begin{equation}
\label{E26}
\beta_e(B,D)\geq \beta_{e}(B\cup\{\bb\},D)+1.
\end{equation}
If $\beta_{e}(B\cup\{\bb\},D)<0$ or $\gamma_e(B\cup\{\bb\},D)>\binom{d+2}{2}-\binom{d-e+2}{2}-1$, then $\tau_e(B\cup\{\bb\},D)=0$, and we are done by (\ref{E24}). Thus we assume that this is not the case, and then  (\ref{E23}) and (\ref{E26}) lead to 
\begin{align}
\label{E27}
\beta_e(B,D)-1&\geq \beta_{e}(B\cup\{\bb\},D)\geq 0\\
\label{E28}
\gamma_e(B,D)&=\gamma_e(B\cup\{\bb\},D)\leq \binom{d+2}{2}-\binom{d-e+2}{2}-1.
\end{align}
If $\gamma_e(B\cup\{\bb\},D)=\binom{d+2}{2}-\binom{d-e+2}{2}-1$, then (\ref{E28}) leads to
\begin{equation}
\label{E29}
\gamma_e(B,D)=\gamma_e(B\cup\{\bb\},D)=\binom{d+2}{2}-\binom{d-e+2}{2}-1,
\end{equation}
and hence
\begin{align*}
\tau_e(B\cup\{\bb\},D)&=\alpha_e(D)+ \beta_e(B\cup\{\bb\},D)+|B\cup\{\bb\}|+2&\Big(\text{by (\ref{E21}),(\ref{E27}),(\ref{E29})}\Big)\\
&\leq \alpha_e(D)+ \beta_e(B,D)+|B|+2&\Big(\text{by (\ref{E27})}\Big)\\
&=\tau_e(B,D).&\Big(\text{by (\ref{E21}),(\ref{E27}),(\ref{E29})}\Big)
\end{align*}
If $\gamma_e(B\cup\{\bb\},D)<\binom{d+2}{2}-\binom{d-e+2}{2}-1$, then (\ref{E28}) leads to
\begin{equation}
\label{E30}
\gamma_e(B,D)=\gamma_e(B\cup\{\bb\},D)<\binom{d+2}{2}-\binom{d-e+2}{2}-1,
\end{equation}
and hence
\begin{align*}
\tau_e(B\cup\{\bb\},D)&=\alpha_e(D)+ \beta_e(B\cup\{\bb\},D)+|B\cup\{\bb\}|+3&\Big(\text{by (\ref{E21}),(\ref{E27}),(\ref{E30})}\Big)\\
&\leq \alpha_e(D)+ \beta_e(B,D)+|B|+3&\Big(\text{by (\ref{E27})}\Big)\\
&=\tau_e(B,D).&\Big(\text{by (\ref{E21}),(\ref{E27}),(\ref{E30})}\Big)
\end{align*}
Thus, in any case, $\tau_e(B\cup\{\bb\},D)\leq \tau_e(B,D)$, and then (\ref{E22}) implies that
\begin{equation}
\label{E31}
\tau_e(B\cup\{\bb\},D)\leq \tau_e(B,D)<\binom{d+2}{2},
\end{equation}
The claim follows from (\ref{E24}) and (\ref{E31}).
 \item[$\bullet$] Assume that $\bb\in D$.    Since $D\setminus\{\bb\}\subseteq D$, we have that 
\begin{equation}
\label{E32}
V_e(D\setminus\{\bb\})\subseteq V_e(D), 
\end{equation} 
 and then 
 \begin{equation}
 \label{E33}
 \alpha_e(D\setminus\{\bb\})\geq \alpha_e(D).
 \end{equation}
   Now we claim that
  \begin{equation}
  \label{E40}
  (V_e(D)\setminus V_e(D\setminus\{\bb\}))\cap \psi_e(B)=\emptyset.
  \end{equation}
   Indeed, if (\ref{E40}) is false, then there is $\ab\in B$ such that $\psi_e(\ab)\in V_e(D)\setminus V_e(D\setminus\{\bb\})$. Since 
   \begin{equation*}
   \dim V_e(D)-\dim V_e(D\setminus\{\bb\})\leq |D|-|D\setminus\{\bb\}|=1, 
   \end{equation*}
   we have that the flat generated by $V_e(D\setminus\{\bb\})$ and $\psi_e(\ab)$ has to be $V_e(D)$, but this flat is precisely $\Fl(V_e(D\setminus\{\bb\})\cup\{\psi_e(\ab)\})=V_e((D\setminus\{\bb\})\cup\{\ab\})$ so 
   \begin{equation}
   \label{E41}
   V_e(D)=V_e((D\setminus\{\bb\})\cup\{\ab\}).
   \end{equation}
  From  (\ref{E41}), we can apply Remark \ref{R20}.iii to get that 
   \begin{equation}
  \label{E42}
   U_e(B,(D\setminus\{\bb\})\cup\{\ab\})\subseteq  U_e(B\cup\{\bb\},D).
  \end{equation}
  Since $(e,D)\in I(B\cup\{\bb\},C)$, we have that $C\nsubseteq U_e(B\cup\{\bb\},D)$. From (\ref{E42}), we get that $C\nsubseteq   U_e(B,(D\setminus\{\bb\})\cup\{\ab\})$ and hence $(e,(D\setminus\{\bb\})\cup\{\ab\})\in I(B,C)$. Thus, in so far as $\bb\not\in U$, we get  that $\bb\not\in  U_e(B,(D\setminus\{\bb\})\cup\{\ab\})$. From (\ref{E21}) and (\ref{E41}), note that $\alpha_e((D\setminus\{\bb\})\cup\{\ab\})\geq 0$ so $\psi_e^{-1}(V_e((D\setminus\{\bb\})\cup\{\ab\}))\subseteq  U_e(B,(D\setminus\{\bb\})\cup\{\ab\})$ and hence $\bb\not\in  \psi_e^{-1}(V_e((D\setminus\{\bb\})\cup\{\ab\}))$. This means that $\psi_e(\bb)\not\in V_e((D\setminus\{\bb\})\cup\{\ab\})$; however,  $\bb\in D$ so $\psi_e(\bb)\in V_e(D)$ but this contradicts (\ref{E41}) and proves (\ref{E40}). Now, from (\ref{E40}), 
  \begin{equation*}
   (\psi_e^{-1}(V_e(D))\setminus \psi_e^{-1}(V_e(D\setminus\{\bb\})))\cap B=\emptyset,
  \end{equation*}
and thus $B\setminus  \psi_e^{-1}(V_e(D\setminus\{\bb\}))=(B\cup\{\bb\})\setminus \psi_e^{-1}(V_e(D))$.  This gives
\begin{equation}
\label{E34}
W_e(B,D\setminus\{\bb\})=W_e(B\cup\{\bb\},D),
\end{equation}
 and therefore 
 \begin{equation}
 \label{E35}
 \beta_e(B,D\setminus\{\bb\})=\beta_e(B\cup\{\bb\},D).
 \end{equation}
 From (\ref{E21}), (\ref{E32}) and (\ref{E34}), we can apply Remark \ref{R20}.ii and we get that
  \begin{equation}
  \label{E36}
   U_e(B,D\setminus\{\bb\})\subseteq  U_e(B\cup\{\bb\},D).
  \end{equation}
  Since $(e,D)\in I(B\cup\{\bb\},C)$, we have that $C\nsubseteq U_e(B\cup\{\bb\},D)$. Thus, from (\ref{E36}), we obtain that $C\nsubseteq  U_e(B,D\setminus\{\bb\})$ and hence $(e,D\setminus \{\bb\})\in I(B,C)$. This implies by assumption that  
  \begin{equation}
\label{E37}
\max\left\{\tau_e(B,D\setminus\{\bb\}),\mu_e(B,D\setminus\{\bb\})\right\}<\binom{d+2}{2}.
\end{equation}
   Since  $(e,D\setminus \{\bb\})\in I(B,C)$ and $\bb\not\in U$, we get  that $\bb\not\in  U_e(B,D\setminus\{\bb\})$. From (\ref{E21}) and (\ref{E33}), we have that $\alpha_e(D\setminus\{\bb\})\geq 0$ so $\psi_e^{-1}(V_e(D\setminus\{\bb\}))\subseteq  U_e(B,D\setminus\{\bb\})$ and hence $\bb\not\in  \psi_e^{-1}(V_e(D\setminus\{\bb\}))$. This means that 
\begin{equation}
\label{E38}
\psi_e(\bb)\not\in V_e(D\setminus\{\bb\}),
\end{equation}  
  and thereby $V_e(D\setminus\{\bb\})$ is a flat properly contained in $V_e(D)$; in particular, 
  \begin{equation}
  \label{E39}
  \alpha_e(D\setminus\{\bb\})\geq \alpha_e(D)+1.
  \end{equation} 
  Then
  \begin{align}
  \label{E43}
  \gamma_e(B\cup\{\bb\},D)=&|V_e(D)\cap \psi_e(B\cup\{\bb\})|\nonumber\\
  =&| (V_e(D)\setminus V_e(D\setminus\{\bb\}))\cap \psi_e(B\cup\{\bb\})|\nonumber\\
  &+| V_e(D\setminus\{\bb\})\cap \psi_e(B\cup\{\bb\})|\nonumber\\
=&1+| V_e(D\setminus\{\bb\})\cap \psi_e(B\cup\{\bb\})|&\Big(\text{by (\ref{E40})}\Big)\nonumber\\
=&1+| V_e(D\setminus\{\bb\})\cap \psi_e(B)|&\Big(\text{by (\ref{E38})}\Big)\nonumber\\
=&1+\gamma_e(B,D\setminus\{\bb\}).
  \end{align}
  Then 
\begin{align}
\label{E44}
\mu_e(B\cup\{\bb\},D)&=\alpha_e(D)+ \gamma_e(B\cup\{\bb\},D)+\binom{d-e+2}{2}&\Big(\text{by (\ref{E21})}\Big)\nonumber\\
&\leq\alpha_e(D\setminus\{\bb\})+ \gamma_e(B,D\setminus\{\bb\})+\binom{d-e+2}{2}&\Big(\text{by (\ref{E39}),(\ref{E43})}\Big)\nonumber\\
&=\mu_e(B,D\setminus\{\bb\})&\Big(\text{by (\ref{E21}),(\ref{E39})}\Big)\nonumber\\
&<\binom{d+2}{2}.&\Big(\text{by (\ref{E37})}\Big)
\end{align}
If $\beta_{e}(B\cup\{\bb\},D)<0$ or $\gamma_e(B\cup\{\bb\},D)>\binom{d+2}{2}-\binom{d-e+2}{2}-1$, then $\tau_e(B\cup\{\bb\},D)=0$, and we are done by (\ref{E44}). Thus suppose that these inequalities are not true and then  (\ref{E35}) and (\ref{E43}) give
\begin{align}
\label{E45}
\beta_e(B,D\setminus\{\bb\})&= \beta_{e}(B\cup\{\bb\},D)\geq 0\\
\label{E46}
\gamma_e(B,D\setminus\{\bb\})+1&=\gamma_e(B\cup\{\bb\},D)\leq \binom{d+2}{2}-\binom{d-e+2}{2}-1.
\end{align}
If $\gamma_e(B\cup\{\bb\},D)=\binom{d+2}{2}-\binom{d-e+2}{2}-1$, then (\ref{E46}) leads to
\begin{equation}
\label{E47}
\gamma_e(B,D\setminus\{\bb\})+1=\gamma_e(B\cup\{\bb\},D)=\binom{d+2}{2}-\binom{d-e+2}{2}-1,
\end{equation}
and hence
\begin{align*}
\tau_e(B\cup\{\bb\},D)&=\alpha_e(D)+ \beta_e(B\cup\{\bb\},D)+|B\cup\{\bb\}|+2&\Big(\text{by (\ref{E21}),(\ref{E45}),(\ref{E47})}\Big)\\
&\leq \alpha_e(D\setminus \{\bb\})+ \beta_e(B,D\setminus \{\bb\})+|B|+2&\Big(\text{by (\ref{E39}),(\ref{E45})}\Big)\\
&\leq \tau_e(B,D\setminus\{\bb\}).&\Big(\text{by (\ref{E21}),(\ref{E45}),(\ref{E47})}\Big)
\end{align*}
If $\gamma_e(B\cup\{\bb\},D)<\binom{d+2}{2}-\binom{d-e+2}{2}-1$, then (\ref{E46}) leads to
\begin{equation}
\label{E48}
\gamma_e(B,D\setminus\{\bb\})+1=\gamma_e(B\cup\{\bb\},D)<\binom{d+2}{2}-\binom{d-e+2}{2}-1,
\end{equation}
and hence
\begin{align*}
\tau_e(B\cup\{\bb\},D)&=\alpha_e(D)+ \beta_e(B\cup\{\bb\},D)+|B\cup\{\bb\}|+3&\Big(\text{by (\ref{E21}),(\ref{E45}),(\ref{E48})}\Big)\\
&\leq \alpha_e(D\setminus \{\bb\})+ \beta_e(B,D\setminus \{\bb\})+|B|+3&\Big(\text{by (\ref{E39}),(\ref{E45})}\Big)\\
&= \tau_e(B,D\setminus\{\bb\}).&\Big(\text{by (\ref{E21}),(\ref{E45}),(\ref{E48})}\Big)
\end{align*}
In any case, $\tau_e(B\cup\{\bb\},D)\leq \tau_e(B,D\setminus\{\bb\})$, and then (\ref{E37}) leads to
\begin{equation}
\label{E49}
\tau_e(B\cup\{\bb\},D)\leq \tau_e(B,D\setminus\{\bb\})<\binom{d+2}{2}.
\end{equation}
The claim follows from (\ref{E44}) and (\ref{E49}).\qedhere
  \end{enumerate}
   \end{proof}
    \begin{lem}
   \label{R23}
     Let $d\in\Nat$, $A,C_0 \in\Pc\left(\Real^2\right)$ and $B_0\in\Pc(A\setminus C_0)$. Take  $B\in \Pc_{\binom{d+2}{2}-3}(A)$ which satisfies the following properties.
     \begin{enumerate}
     \item[i)]$\dim \psi_d(B)=\binom{d+2}{2}-4$.
     \item[ii)] $B\supseteq B_0$.
     \item[iii)]$B\cap C_0=B\setminus B_0$.
     \item[iv)]For all $(e,D)\in I(B,C_0)$, we have that $\max\{\tau_e(B,D),\mu_e(B,D)\}<\binom{d+2}{2}$.
     \item[v)]For all $(e,D)\in ([1,d-1]\times \Pc(B))\setminus I(B,C_0)$ with $\alpha_e(D)\geq 0$, we have that  $\gamma_e(B,D)<\binom{d+2}{2}-\binom{d-e+2}{2}$ and   $\beta_e(B,D)<0$.
     \end{enumerate}
     Then $B\in\Nc_d(A,B_0,C_0)$. 
   \end{lem}
   \begin{proof}
Because of the assumptions  i), ii) and iii), it suffices to show that $B$ satisfies the conditions ii), iii) and iv) of the definition of $\Nc_d(A)$.  First we show that for all $e\in[1,d-1]$ and $C\in \Cc_e$, we have that
  \begin{equation}
  \label{E50}
  |C\cap B|<\binom{d+2}{2}-\binom{d-e+2}{2}.
  \end{equation}
  Fix $p(x,y)\in\Real_e[x,y]$ such that $C=\Zc(p(x,y))$, and write $H:=\tau_e([p(x,y)])$ and $D:=B\cap \psi_e^{-1}(H)$. Hence $V_e(D)\subseteq H$ and   Remark \ref{R12}.ii implies that $C=\psi_e^{-1}(H)$  so
  \begin{equation}
  \label{E51}
  |C\cap B|=|H\cap \psi_e(B)|=|V_e(D)\cap \psi_e(B)|=\gamma_e(B,D).
  \end{equation}
  In so far as $\dim V_e(D)\leq \dim H=\binom{e+2}{2}-2$, we get that $\alpha_e(D)\geq 0$, and therefore 
  \begin{equation}
  \label{E52}
 \mu_e(B,D)=\alpha_e(D)+ \gamma_e(B,D)+\binom{d-e+2}{2}.
  \end{equation}
If $(e,D)\in I(B,C_0)$, then  $\mu_e(B,D)<\binom{d+2}{2}$ by iv). Insomuch as $\alpha_e(D)\geq 0$, we obtain from (\ref{E52}) that $ \gamma_e(B,D)<\binom{d+2}{2}-\binom{d-e+2}{2}$. If $(e,D) \not\in I(B,C_0)$, then  $\gamma_e(B,D)<\binom{d+2}{2}-\binom{d-e+2}{2}$ by v). Hence, in any case, 
  \begin{equation}
  \label{E53}
  \gamma_e(B,D)<\binom{d+2}{2}-\binom{d-e+2}{2},
  \end{equation}
   and thus (\ref{E50}) follows from (\ref{E51}) and (\ref{E53}).
   
   Now take $e\in[1,d-1]$  and $C\in \Cc_e$ such that 
\begin{equation}
\label{E54}
|C\cap B|=\binom{d+2}{2}-\binom{d-e+2}{2}-1,
\end{equation}   
   and we will show that 
   \begin{equation}
   \label{E55}
   \dim \psi_{d-e}(B\setminus C)= \binom{d-e+2}{2}-3.
   \end{equation}
   Since 
   \begin{equation*}
  |B\setminus C|=|B|-|C\cap B|=\binom{d-e+2}{2}-2, 
   \end{equation*}
   we have that 
   \begin{equation}
   \label{E56}
      \dim \psi_{d-e}(B\setminus C)\leq   |B\setminus C|-1=\binom{d-e+2}{2}-3.
   \end{equation}
       Fix $p(x,y)\in\Real_e[x,y]$ such that $C=\Zc(p(x,y))$, and write $H:=\tau_e([p(x,y)])$ and $D:=B\cap \psi_e^{-1}(H)$. Hence $V_e(D)\subseteq H$ and   Remark \ref{R12}.ii implies that $ \psi_e(C)= H\cap \psi_e(\Real^2)$  so
  \begin{equation*}
   |C\cap B|=|\psi_e(C)\cap \psi_e(B)|=|H\cap \psi_e(B)|=|V_e(D)\cap \psi_e(B)|=\gamma_e(B,D), 
  \end{equation*}
   and then (\ref{E54}) implies that
   \begin{equation}
   \label{E57}
   \gamma_e(B,D)=\binom{d+2}{2}-\binom{d-e+2}{2}-1.
   \end{equation}
    In so far as $\dim V_e(D)\leq \dim H=\binom{e+2}{2}-2$, we get that
    \begin{equation}
    \label{E58}
    \alpha_e(D)\geq 0.
    \end{equation}
    Since $D=B\cap \psi_e^{-1}(H)=B\cap C$, we have that $B\setminus \psi^{-1}_e(V_e(D))=B\setminus C$ so
    \begin{equation*}
    W_e(B,D)=\Fl(\psi_{d-e}(B\setminus \psi^{-1}_e(V_e(D)))=\Fl(\psi_{d-e}(B\setminus C)),
    \end{equation*} 
   and then
    \begin{equation}
    \label{E59}
    \dim W_e(B,D)=\dim \psi_{d-e}(B\setminus C).
    \end{equation}
    From (\ref{E56}) and (\ref{E59}),
    \begin{equation}
    \label{E60}
    \beta_e(B,D)\geq 0.
    \end{equation}
    From (\ref{E57}), (\ref{E58}) and (\ref{E60}), 
    \begin{equation}
    \label{E61}
    \tau_e(B,D)=\alpha_e(D)+ \beta_e(B,D)+|B|+2.
    \end{equation}
  From (\ref{E58}) and (\ref{E60}), $\min\{\alpha_e(D),\beta_e(B,D)\}\geq 0$; then  v) implies  that $(e,D)\in I(B,C_0)$. Thus iv) implies that $\tau_e(B,D)<\binom{d+2}{2}$, and then  (\ref{E58}) and (\ref{E61}) gives $\beta_e(B,D)<1$. Now, from (\ref{E60}), $\beta_e(B,D)=0$ and then (\ref{E59}) yields (\ref{E55}).
    
   Now take $e\in[1,d-1]$  and $C\in \Cc_e$ such that 
\begin{equation}
\label{E62}
|C\cap B|<\binom{d+2}{2}-\binom{d-e+2}{2}-1,
\end{equation}   
   and we will prove that 
   \begin{equation}
   \label{E63}
   \dim \psi_{d-e}(B\setminus C)>\binom{d-e+2}{2}-3.
   \end{equation}
  Fix $p(x,y)\in\Real_e[x,y]$ such that $C=\Zc(p(x,y))$, and write $H:=\tau_e([p(x,y)])$ and $D:=B\cap \psi_e^{-1}(H)$. Hence $V_e(D)\subseteq H$ and   Remark \ref{R12}.ii implies that $ \psi_e(C)= H\cap \psi_e(\Real^2)$  so
  \begin{equation*}
   |C\cap B|=|\psi_e(C)\cap \psi_e(B)|=|H\cap \psi_e(B)|=|V_e(D)\cap \psi_e(B)|=\gamma_e(B,D), 
  \end{equation*}
   and thus (\ref{E62}) gives
   \begin{equation}
   \label{E64}
   \gamma_e(B,D)<\binom{d+2}{2}-\binom{d-e+2}{2}-1.
   \end{equation}
   Since  $\dim V_e(D)\leq \dim H=\binom{e+2}{2}-2$, we get that
    \begin{equation}
    \label{E65}
    \alpha_e(D)\geq 0.
    \end{equation}
    Inasmuch as $D=B\cap \psi_e^{-1}(H)=B\cap C$, notice that $B\setminus \psi_e^{-1}(V_e(D))=B\setminus C$ so
    \begin{equation*}
    W_e(B,D)=\Fl(\psi_{d-e}(B\setminus \psi^{-1}_e(V_e(D)))=\Fl(\psi_{d-e}(B\setminus C)),
    \end{equation*} 
    and thus
    \begin{equation}
    \label{E66}
    \dim W_e(B,D)=\dim \psi_{d-e}(B\setminus C).
    \end{equation}
         We claim that 
    \begin{equation}
    \label{E67}
    \beta_e(B,D)<0.
    \end{equation}
   Indeed, if (\ref{E67}) is false,  then $\beta_e(B,D)\geq 0$  and therefore $\min\{\alpha_e(D),\beta_e(B,D)\}\geq 0$ by (\ref{E65}). Hence v) would imply that $(e,D)\in I(B,C_0)$ and thereby $\tau_e(B,D)<\binom{d+2}{2}$ by iv). However, (\ref{E64}), (\ref{E65}) and $\beta_e(B,D)\geq 0$ give
    \begin{equation*}
        \tau_e(B,D)=\alpha_e(D)+ \beta_e(B,D)+|B|+3\geq \binom{d+2}{2},
    \end{equation*}
     which is impossible. This means that (\ref{E67}) is true, and thereby (\ref{E63}) is a consequence of (\ref{E66}) and (\ref{E67}).
         \end{proof}
         Recall that $A\in\Pc(\Real^2)$ is $d$-regular if $|A\cap C|< \frac{1}{2^{2^{3d+8}}}|A|$ for all $C\in \Cc_{\leq d}$.
         \begin{lem}
      \label{R24}
      Let $d\in \Zet$  with $d>1$ and $A$ be $d$-regular. Then
      \begin{equation*}
      |\Nc_d(A,\emptyset,\Real^2)|\geq \frac{1}{2^{d+3}!}|A|^{\binom{d+2}{2}-3}.
      \end{equation*}
      \end{lem}
      \begin{proof}
      Set $g:=\binom{d+2}{2}-3$. Throughout this proof,  for any $\obb\in A^g$, its $i$-th entry will be denoted by $\bb_i$, i.e. $\obb=(\bb_1,\bb_2,\ldots, \bb_g)$. For any $\obb\in A^g$, write  $B_0(\obb):=\emptyset$. Now,  for all $i\in[1,g]$ and $\obb\in A^g$, set
      \begin{align*}
      B_i(\obb)&:=\{\bb_1,\bb_2,\ldots,\bb_i\}\\
      U_i(\obb)&:=\bigcup_{(e,D)\in I(B_{i-1}(\obb),\Real^2)}U_e(B_{i-1}(\obb),D).
      \end{align*}
      For all $i\in[1,g]$ and $\obb\in A^g$, notice that $| B_i(\obb)|\leq g<\binom{d+2}{2}-1$. Then Lemma \ref{R21}.iv implies that for all $e\in [1,d-1]$ and $D\in \Pc\left(B_{i}(\obb)\right)$, we get that $U_e\left(B_{i}(\obb),D\right)$ is contained in the union of $3$ curves in $\Cc_{\leq d}$; in particular, $U_e\left(B_{i}(\obb),D\right)$ cannot contain $\Real^2$, and therefore
      \begin{equation}
      \label{E68}
       I\left(B_{i}(\obb),\Real^2\right)=[1,d-1]\times \Pc\left(B_{i}(\obb)\right).
      \end{equation}
       For each $i\in[1,g]$, write $R_i:=\left\{\obb\in A^g:\; \bb_i\in U_i(\obb)\right\}$. The first step in the proof is to show that for all $i\in[1,g]$,
      \begin{equation}
      \label{E69}
      |R_i|\leq \frac{1}{2g}|A|^g.
      \end{equation}
      For each $i\in[1,g], E\in \Pc([1,i])$ and $\obb\in A^g$, set $E(\obb):=\{\bb_j:\;j\in E\}$. Notice that $\Pc\left(B_{i}(\obb)\right)=\{E(\obb):\;E\in \Pc([1,i])\}$. Now, for each $i\in[1,g], E\in \Pc([1,i])$ and $e\in[1,d-1]$, write $R_{i,e,E}:=\left\{\obb\in A^g:\; \bb_i\in U_e(B_{i-1}(\obb),E(\obb))\right\}$. From (\ref{E68}), notice that for all $i\in [1,g]$, we get that $R_i=\bigcup_{E\in \Pc([1,i])}\bigcup_{e=1}^{d-1}R_{i,e,E}$ and then 
      \begin{equation}
      \label{E70}
       |R_i|\leq \sum_{E\in \Pc([1,i])}\sum_{e=1}^{d-1}|R_{i,e,E}|.
      \end{equation}
       For each $i\in[1,g], E\in \Pc([1,i])$ and $\obb\in A^g$, there exist $C_1,C_2,C_3\in\Cc_{\leq d}$ such that $U_e(B_{i-1}(\obb),E(\obb))\subseteq C_1\cup C_2\cup C_3$ by Lemma \ref{R21}.iv; then 
\begin{equation*}
A\cap U_e(B_{i-1}(\obb),E(\obb))\subseteq A\cap (C_1\cup C_2\cup C_3),
\end{equation*}       
      and,  since $A$ is $d$-regular, 
       \begin{equation}
       \label{E71}
       |A\cap U_e(B_{i-1}(\obb),E(\obb))|\leq \sum_{j=1}^3|A\cap C_i|<\frac{3}{2^{2^{3d+8}}}|A|\leq \frac{1}{2^{2^{3d+6}}}|A|. 
              \end{equation}
             Furthermore, in so far as $g=\binom{d+2}{2}-3<2^{d+2}$, we have that $dg2^{g+1}\leq 2^{2^{3d+6}}$. Hence, inasmuch as the $i$-th entry of $\obb\in R_{i,e,E}$ has to be in $A\cap U_e(B_{i-1}(\obb),E(\obb))$, (\ref{E71}) leads to   
              \begin{equation}
              \label{E72}
              |R_{i,e,E}|\leq |A|^{g-1}|A\cap U_e(B_{i-1}(\obb),E(\obb))| \leq  \frac{1}{2^{2^{3d+6}}}|A|^g\leq  \frac{1}{gd2^{g+1}}|A|^g. 
              \end{equation}
             Since $|\Pc([1,i])\times [1,d-1]|\leq d2^g$, we have that (\ref{E70}) and (\ref{E72}) imply (\ref{E69}).
             
             Set $T:=\left\{\obb\in A^g:\; \bb_i\not\in U_i(\obb) \text{ for all }i\in[1,g]\right\}$ so $T=A^g\setminus \bigcup_{i=1}^gR_i$. From (\ref{E69}), 
             \begin{equation}
             \label{E73}
             |T|=\left|A^g\setminus \bigcup_{i=1}^gR_i\right|\geq |A|^g-\sum_{i=1}^g|R_i|\geq \frac{1}{2}|A|^g.
             \end{equation}
                The next step is to show that the map
              \begin{equation*}
              \phi:T\longrightarrow \Pc(A),\qquad \phi(\obb)=B_g(\obb)
              \end{equation*}
               satisfies that
      \begin{equation}
      \label{E74}
      \phi(T)\subseteq \Nc_d(A, \emptyset,\Real^2).
\end{equation} 
We claim that for all $\obb\in T$, $i\in[1,g]$ and $(e,D)\in [1,d-1]\times \Pc\left(B_{i}(\obb)\right)$, we get that 
\begin{equation}
\label{E75}
\max\left\{\tau_e(B_i(\obb),D),\mu_e(B_i(\obb),D)\right\}<\binom{d+2}{2}.
\end{equation}    
We show (\ref{E75}) by induction on $i\in[1,g]$. For $i=1$, notice  that $B_1(\obb)=\{\bb_1\}$   and then $\alpha_e(D)\leq \binom{e+2}{2}-2, \beta_e(B_1(\obb),D)\leq \binom{d-e+2}{2}-3$  and $\gamma_e(B_1(\obb),D)\leq 1$ so
\begin{equation*}
\max\left\{\tau_e(B_1(\obb),D),\mu_e(B_1(\obb),D)\right\}\leq\binom{e+2}{2}+\binom{d-e+2}{2}-1<\binom{d+2}{2}, 
\end{equation*}
and then (\ref{E75}) follows in this case. Assume that (\ref{E75}) holds for $i-1\geq 1$ and we show it for $i$. Since $\obb\in T$, note that $\bb_i\in A\setminus U_i(\obb)$. By induction,
\begin{equation*}
\max\left\{\tau_e(B_{i-1}(\obb),D),\mu_e(B_{i-1}(\obb),D)\right\}<\binom{d+2}{2}
\end{equation*}
for all  $(e,D)\in [1,d-1]\times \Pc\left(B_{i-1}(\obb)\right)=  I\left(B_{i-1}(\obb),\Real^2\right)$ so $B_{i-1}(\obb)$ and $\bb_i\in A\setminus U_i(\obb)$ satisfy the assumptions of Lemma \ref{R22}, and hence this lemma implies that $B_{i-1}(\obb)\cup\{\bb_i\}=B_{i}(\obb)$ satisfies (\ref{E75}) completing the induction. In particular, (\ref{E75}) holds for $i=g$ so 
\begin{equation}
\label{E76}
\max\left\{\tau_e(B_g(\obb),D),\mu_e(B_g(\obb),D)\right\}<\binom{d+2}{2}
\end{equation}
 for all $\obb\in T$ and $(e,D)\in [1,d-1]\times \Pc\left(B_{g}(\obb)\right)$.  For all $\obb\in T$, we get that $\bb_i\not\in U_i(\obb)$; in particular, $\bb_i\not\in\psi^{-1}_d(V_d(B_{i-1}(\obb)))$. This means that the flats $V_d(B_1(\obb))\subsetneq V_d(B_2(\obb))\subsetneq\ldots \subsetneq V_d(B_g(\obb))$ are contained properly so 
 \begin{equation}
 \label{E77}
\dim V_d(B_g(\obb))\geq g-1.
 \end{equation}
 On the other hand, 
 \begin{equation*}
 \dim V_d(B_g(\obb))=\dim \Fl(\psi_d(B_g(\obb)))\leq |\psi_d(B_g(\obb))|-1=|B_g(\obb)|-1=g-1
 \end{equation*}
 so (\ref{E77}) leads to 
  \begin{equation}
 \label{E78}
 \dim \psi_d(B_g(\obb))=\dim V_d(B_g(\obb))=g-1=\binom{d+2}{2}-4.
 \end{equation}
 We want to apply Lemma \ref{R23} with $B=B_g(\obb)$, $B_0=\emptyset$ and $C_0=\Real^2$, and this can be done because its assumptions are satisfied (indeed: i) holds by  (\ref{E78}); ii) and iii) are trivial; iv) holds by (\ref{E68}) and (\ref{E76}); v) never happens by (\ref{E68})).  Hence Lemma \ref{R23} implies that $B_g(\obb)\in\Nc_d(A, \emptyset,\Real^2)$ and it proves (\ref{E74}).
 
 To conclude the proof, note that for each $B\in\phi(T)$, we have that $\phi^{-1}(B)\subseteq \{(\bb_1,\ldots,\bb_g)\in A^g:\;\{\bb_1,\ldots,\bb_g\}=B\}$. Therefore, for all $B\in \Nc_d(A, \emptyset,\Real^2)$, we have that $|\phi^{-1}(B)|\leq g!$,  and then (\ref{E73}) and (\ref{E74}) yield that
 \begin{equation}
 \label{E79}
 |\Nc_d(A,\emptyset,\Real^2)|\geq \frac{1}{g!}|T|\geq  \frac{1}{g!2}|A|^g.
\end{equation} 
In so far as $g=\binom{d+2}{2}-3<2^{d+2}$, we have that $g!2\leq 2^{d+3}!$  and then (\ref{E79}) implies the claim of the lemma. 
      \end{proof}
       \begin{lem}
      \label{R25}
Let $d,f\in\Zet$ be such that $d\geq 3$ and $f\in[0,d-1]$, and $B\in\Pc_{\binom{f+2}{2}}\left(\Real^2\right)$ be such that  $B$ is not contained in an element of $\Cc_{\leq f}$. Then, for all $e\in[1,d-1]$ and $D\in\Pc(B)$, 
 \begin{equation*}
 \max\{\tau_e(B,D),\mu_e(B,D)\}<\binom{d+2}{2}.
 \end{equation*}
  \end{lem}
  \begin{proof}
  Fix   $e\in[1,d-1]$ and $D\in\Pc(B)$. If $\alpha_e(D)<0$, then $\tau_e(B,D)=\mu_e(B,D)=0$. Hence, from now on, we assume that
  \begin{equation}
  \label{E80}
  \alpha_e(D)\geq 0.
  \end{equation}
  First we show that
  \begin{equation}
  \label{E81}
  \mu_e(B,D)<\binom{d+2}{2}.
  \end{equation}
  If $e\geq f$, then Lemma \ref{R15}.i leads to
  \begin{equation*}
  |\psi_e(B)\cap V_e(D)|\leq 1+\dim V_e(D), 
  \end{equation*}
   and hence 
   \begin{align*}
   |\psi_e(B)\cap V_e(D)|+\alpha_e(D)&\leq 1+\dim V_e(D)+\alpha_e(D)\\
   &=\binom{e+2}{2}-1\\
   &\leq \binom{d+1}{2}-\binom{d-e+1}{2}.&\Big(\text{since $d-1\geq e$}\Big)
   \end{align*}
   If $e< f$, then Lemma \ref{R15}.ii leads to
  \begin{equation*}
  |\psi_e(B)\cap V_e(D)|\leq \binom{f+2}{2}-\binom{f-e+2}{2}-\binom{e+2}{2}+2+\dim V_e(D)
  \end{equation*}
 so
   \begin{align*}
   |\psi_e(B)\cap V_e(D)|+\alpha_e(D)&\leq \binom{f+2}{2}-\binom{f-e+2}{2}\\
    &\leq \binom{d+1}{2}-\binom{d-e+1}{2}.&\Big(\text{since $d-1\geq f$}\Big)
   \end{align*}
   Thus, in any case, 
   \begin{equation}
   \label{E82}
   \gamma_e(B,D)+\alpha_e(D)=|\psi_e(B)\cap V_e(D)|+\alpha_e(D)\leq \binom{d+1}{2}-\binom{d-e+1}{2}.
   \end{equation}
   We get (\ref{E81}) as follows.
   \begin{align*}
   \mu_e(B,D)&= \gamma_e(B,D)+\alpha_e(D)+\binom{d-e+2}{2}&\Big(\text{by (\ref{E80})}\Big)\\
   &=\binom{d+1}{2}-\binom{d-e+1}{2}+\binom{d-e+2}{2}&\Big(\text{by (\ref{E82})}\Big)\\
   &<\binom{d+2}{2}.
   \end{align*}
   Now we prove that
    \begin{equation}
  \label{E83}
  \tau_e(B,D)<\binom{d+2}{2}.
  \end{equation}
  If $\beta_e(B,D)<0$, then $\tau_e(B,D)=0$ and in this case we are done by (\ref{E81}). Thus we assume that 
  \begin{equation}
  \label{E84}
  \beta_e(B,D)\geq 0.
  \end{equation}
  Remark \ref{R20}.i gives
  \begin{equation}
  \label{E85}
    \left|B\cap \psi_e^{-1}(V_e(D))\right|+ \left|B\cap \psi_{d-e}^{-1}(W_e(B,D))\right|\geq |B|=\binom{f+2}{2}.
   \end{equation}
   From (\ref{E80}) and   (\ref{E84}),  Remark \ref{R20}.iv  yields that
   \begin{equation}
   \label{E86}
   \tau_e(B,D)\leq \alpha_e(D)+\beta_e(B,D)+\binom{f+2}{2}+3.
\end{equation}    
   We divide  the conclusion of the proof of (\ref{E83}) into four cases.
   \begin{enumerate}
   \item[$\bullet$]Assume that $e\geq f$ and $d-e\geq f$. We apply Lemma \ref{R15}.i to $V_e(D)$ and $W_e(B,D)$ to get 
   \begin{align}
   \label{E87}
   \alpha_e(D)&\leq \binom{e+2}{2}-1-|\psi_e(B)\cap V_e(D)|\\
   \label{E88}
    \beta_e(B,D)&\leq \binom{d-e+2}{2}-2-|\psi_{d-e}(B)\cap W_e(B,D)|.
   \end{align}
   From (\ref{E85}), (\ref{E87}) and (\ref{E88}), we get that
   \begin{equation}
   \label{E89}
    \alpha_e(D)+\beta_e(B,D)\leq \binom{e+2}{2}+\binom{d-e+2}{2}-3-\binom{f+2}{2}.
   \end{equation}
   Hence 
   \begin{align*}
   \tau_e(B,D)&\leq \alpha_e(D)+\beta_e(B,D)+\binom{f+2}{2}+3&\Big(\text{by (\ref{E86})}\Big)\\
   &\leq \binom{e+2}{2}+\binom{d-e+2}{2}&\Big(\text{by (\ref{E89})}\Big)\\
   &<\binom{d+2}{2}.&\Big(\text{since $d\geq 3$}\Big)
   \end{align*}
   \item[$\bullet$]Assume that $e\geq f$ and $d-e< f$. We apply Lemma \ref{R15}.i to $V_e(D)$ and Lemma \ref{R15}.ii to $W_e(B,D)$ so
   \begin{align}
   \label{E90}
   \alpha_e(D)&\leq \binom{e+2}{2}-1-|\psi_e(B)\cap V_e(D)|\\
   \label{E91}
    \beta_e(B,D)&\leq \binom{f+2}{2}-\binom{f-(d-e)+2}{2}-1-|\psi_{d-e}(B)\cap W_e(B,D)|.
   \end{align}
   We have two subcases.
   \begin{enumerate}
   \item[$\star$]Suppose that $e=d-1$. Then 
\begin{equation*}
\beta_e(B,D)=\binom{3}{2}-3-\dim W_e(B,D)=-\dim W_e(B,D)
\end{equation*}   
    so (\ref{E84}) implies that 
   \begin{equation}
   \label{E92}
   \beta_e(B,D)=\dim W_e(B,D)=0.
   \end{equation}
   Since $\dim W_e(B,D)=0$, we have that $W_e(B,D)$ is a point and hence (\ref{E85}) leads to
   \begin{equation}
   \label{E93}
    \left|B\cap \psi_e^{-1}(V_e(D))\right|\geq \binom{f+2}{2}-1.
   \end{equation}
    From (\ref{E90}), (\ref{E92}) and (\ref{E93}), we get that
   \begin{equation}
   \label{E94}
    \alpha_e(D)+\beta_e(B,D)=\alpha_e(D)\leq \binom{e+2}{2}-\binom{f+2}{2}.
   \end{equation}
   Hence 
   \begin{align*}
   \tau_e(B,D)&\leq \alpha_e(D)+\beta_e(B,D)+\binom{f+2}{2}+3&\Big(\text{by (\ref{E86})}\Big)\\
   &\leq \binom{e+2}{2}+3&\Big(\text{by (\ref{E94})}\Big)\\
   &<\binom{d+2}{2}.&\Big(\text{since $e+1=d\geq 3$}\Big)
   \end{align*}
    \item[$\star$]Suppose that $e<d-1$.  From (\ref{E85}), (\ref{E90}) and (\ref{E91}), we get that
   \begin{equation}
   \label{E95}
    \alpha_e(D)+\beta_e(B,D)\leq \binom{e+2}{2}-\binom{f-(d-e)+2}{2}-2.
   \end{equation}
   Thus
   \begin{align*}
   \tau_e(B,D)&\leq \alpha_e(D)+\beta_e(B,D)+\binom{f+2}{2}+3&\Big(\text{by (\ref{E86})}\Big)\\
   &\leq \binom{e+2}{2}-\binom{f-(d-e)+2}{2}+\binom{f+2}{2}+1&\Big(\text{by (\ref{E95})}\Big)\\
    &<\binom{d+2}{2}.&\Big(\text{since $f,e+1\leq d-1$}\Big)\\
     \end{align*}
   \end{enumerate}
   \item[$\bullet$]Assume that $e<f$ and $d-e\geq  f$. We apply Lemma \ref{R15}.ii to $V_e(D)$ and Lemma \ref{R15}.i to $W_e(B,D)$ so
   \begin{align}
   \label{E96}
   \alpha_e(D)&\leq  \binom{f+2}{2}-\binom{f-e+2}{2}-|\psi_e(B)\cap V_e(D)|\\
   \label{E97}
    \beta_e(B,D)&\leq \binom{d-e+2}{2}-2-|\psi_{d-e}(B)\cap W_e(B,D)|.
   \end{align}
   We have three subcases.
   \begin{enumerate}
    \item[$\star$]Suppose that $e=1$ and $\alpha_e(D)>0$. Since $e=1$,  
\begin{equation}
\label{E98}
\alpha_e(D)=\binom{3}{2}-2-\dim V_e(D)=1-\dim V_e(D)
\end{equation}   
    so  the assumption  $\alpha_e(D)>0$ leads to
   \begin{equation}
   \label{E99}
   \alpha_e(D)-1=\dim V_e(D)=0.
   \end{equation}
    Since $\dim V_e(D)=0$, we have that $V_e(D)$ is just  a point and thereby (\ref{E85}) gives
   \begin{equation}
   \label{E100}
    \left|B\cap \psi_{d-e}^{-1}(W_e(B,D))\right|\geq \binom{f+2}{2}-1.
   \end{equation}
    From (\ref{E97}), (\ref{E99}) and (\ref{E100}), we get that
   \begin{equation}
   \label{E101}
    \alpha_e(D)+\beta_e(B,D)=1+\beta_e(B,D)\leq \binom{d-e+2}{2}-\binom{f+2}{2}.
   \end{equation}
   Hence 
   \begin{align*}
   \tau_e(B,D)&\leq \alpha_e(D)+\beta_e(B,D)+\binom{f+2}{2}+3&\Big(\text{by (\ref{E86})}\Big)\\
   &\leq \binom{d-e+2}{2}+3&\Big(\text{by (\ref{E101})}\Big)\\
   &<\binom{d+2}{2}.&\Big(\text{since $d\geq 3$}\Big)
   \end{align*}
    \item[$\star$]Suppose that $e=1$ and $\alpha_e(D)=0$. Proceeding as in (\ref{E98}), we conclude that 
   \begin{equation}
   \label{E102}
   1+\alpha_e(D)=\dim V_e(D)=1.
   \end{equation}
   If $\gamma_e(B,D)>\binom{d+2}{2}-\binom{d-e+2}{2}-1=d$, then $\tau_e(B,D)=0$. Thus we assume that $\gamma_e(B,D)\leq d$. If $\gamma_e(B,D)=d$, then (\ref{E85}) implies that 
   \begin{equation*}
   |\psi_{d-e}(B)\cap W_e(B,D)|\geq \binom{f+2}{2}-\gamma_e(B,D)=\binom{f+2}{2}-d,
   \end{equation*}
    and then (\ref{E97}) and (\ref{E102}) lead to
    \begin{equation*}
    \tau_e(B,D)=\alpha_e(D)+\beta_e(B,D)+\binom{f+2}{2}+2\leq \binom{d-e+2}{2}+d.
    \end{equation*}
    If $\gamma_e(B,D)\leq d-1$, then (\ref{E85}) gives  
   \begin{equation*}
   |\psi_{d-e}(B)\cap W_e(B,D)|\geq \binom{f+2}{2}-\gamma_e(B,D)\geq \binom{f+2}{2}-d+1,
   \end{equation*}
    and then (\ref{E97}) and (\ref{E102}) imply that
    \begin{equation*}
    \tau_e(B,D)=\alpha_e(D)+\beta_e(B,D)+\binom{f+2}{2}+3\leq \binom{d-e+2}{2}+d.
    \end{equation*}
    In any case, 
          \begin{align*}
   \tau_e(B,D) &\leq \binom{d+1}{2}+d<\binom{d+2}{2}.
   \end{align*}
   \item[$\star$]Suppose that $e>1$. From (\ref{E85}), (\ref{E96}) and (\ref{E97}), we get that
   \begin{equation}
   \label{E103}
    \alpha_e(D)+\beta_e(B,D)\leq \binom{d-e+2}{2}-\binom{f-e+2}{2}-2.
   \end{equation}
   Thus
   \begin{align*}
   \tau_e(B,D)&\leq \alpha_e(D)+\beta_e(B,D)+\binom{f+2}{2}+3&\Big(\text{by (\ref{E86})}\Big)\\
   &\leq \binom{d-e+2}{2}-\binom{f-e+2}{2}+\binom{f+2}{2}+1&\Big(\text{by (\ref{E103})}\Big)\\
   &<\binom{d+2}{2}.&\Big(\text{since $e>1$}\Big)\\
     \end{align*}
   \end{enumerate}
   \item[$\bullet$]Assume that $e< f$ and $d-e< f$. The definitions of $W_e(B,D)$ and $V_e(D)$ lead to  $W_{e}(B,D)\supseteq \psi_{d-e}(B\setminus\psi^{-1}_{e}(V_e(D)))$  and $V_{e}(D)\supseteq \psi_{e}(B\setminus\psi^{-1}_{d-e}(W_e(B,D)))$. Then  we get that $\dim W_{e}(B,D)\geq \\\dim\psi_{d-e}(B\setminus\psi^{-1}_{e}(V_e(D)))$  and $\dim V_{e}(D)\geq\dim \psi_{e}(B\setminus\psi^{-1}_{d-e}(W_e(B,D)))$. Hence, as a consequence of Lemma \ref{R15}.iii applied to $V_e(D)$ and $W_e(B,D)$, we get that 
   \begin{align}
   \label{E104}
   \alpha_e(D)&\leq \binom{e+2}{2}-1-\binom{f-(d-e)+2}{2}\\
   \label{E105}
    \beta_e(B,D)&\leq \binom{d-e+2}{2}-\binom{f-e+2}{2}-2.
   \end{align}
   Using that $\binom{g+2}{2}=\sum_{i=1}^{g+1}i$ for any $g\in\Natz$, it is proven easily that 
   \begin{equation}
   \label{EE105}
   \binom{e+2}{2}+\binom{d-e+2}{2}+\binom{f+2}{2}-\binom{f-(d-e)+2}{2}-\binom{f-e+2}{2}<\binom{d+2}{2}.
   \end{equation}
     Hence 
   \begin{align*}
   \tau_e(B,D)\leq & \alpha_e(D)+\beta_e(B,D)+\binom{f+2}{2}+3&\Big(\text{by (\ref{E86})}\Big)\\
   \leq &\binom{e+2}{2}+\binom{d-e+2}{2}+\binom{f+2}{2}&\\
   &-\binom{f-(d-e)+2}{2}-\binom{f-e+2}{2}&\Big(\text{by (\ref{E104}),(\ref{E105})}\Big)\\
   <&\binom{d+2}{2}.&\Big(\text{by (\ref{EE105})}\Big)
   \end{align*}
   \end{enumerate}
   This concludes the proof of (\ref{E83}). The claim is a consequence of (\ref{E81}) and (\ref{E83}).
  \end{proof}
   \begin{lem}
   \label{R26}
    Let $d\in\Nat$, $f\in[0,d-1]$,  $C_0\in\Cc_{d-f}$ be irreducible and $B_0\in\Pc_{\binom{f+2}{2}}(\Real^2\setminus C_0)$ be such that $B_0$ is not contained in an element of $\Cc_{\leq f}$. Take  $B\in \Pc_{\binom{d+2}{2}-3}(\Real^2)$ such that   $B\supseteq B_0$ and  $B\cap C_0=B\setminus B_0$, and also take  $(e,D)\in ([1,d-1]\times \Pc(B))\setminus I(B,C_0)$ such that $\alpha_e(D)\geq 0$.  Then $\beta_e(B,D)<0$ and  $\gamma_e(B,D)<\binom{d+2}{2}-\binom{d-e+2}{2}-2$.
       \end{lem}
     \begin{proof}
Since $\alpha_e(D)\geq 0$, 
\begin{align*}
U_e(B,D)&=\left\{ \begin{array}{ll}
 \psi_d^{-1}(V_d(B))\cup\psi_e^{-1}(V_e(D)) & \text{if $\beta_e(B,D)<0$}\\
&\\
 \psi_d^{-1}(V_d(B))\cup\psi_e^{-1}(V_e(D))\cup \psi_{d-e}^{-1}(W_e(B,D)) & \text{if $\beta_e(B,D)\geq 0$}.\end{array} \right.
   \end{align*}
In so far as $|B|<\binom{d+2}{2}-1$, Lemma \ref{R21}.iii implies that there is $C_1\in \Cc_{\leq d}$ such that $ \psi_d^{-1}(V_d(B))\subseteq C_1$. Since $\alpha_e(D)\geq 0$, Lemma \ref{R21}.i yields the existence of $C_2\in \Cc_{\leq d}$ such that $\psi_e^{-1}(V_e(D))\subseteq C_2$. Now, if $\beta_e(B,D)<0$, write $C_3=\emptyset$, and if  $\beta_e(B,D)\geq 0$, we fix $C_3\in \Cc_{\leq d-e}$ such that $\psi_{d-e}^{-1}(W_e(B,D))\subseteq C_3$ (which exists by Lemma \ref{R21}.ii). Thus, in any case, 
 \begin{equation}
 \label{E106}
 U_e(B,D)\subseteq C_1\cup C_2\cup C_3 .
 \end{equation}
 In so far as $(e,D)\not \in I(B,C_0)$, note that $C_0\subseteq U_e(B,D)$. From (\ref{E106}), we have that $C_0$ is an irreducible component of $C_1$ or $C_2\cup C_3$. We claim that 
  \begin{equation}
  \label{E107}
  C_0\subseteq C_2\cup C_3.
  \end{equation}
  If (\ref{E107}) is false, then $C_0$ is a component of $C_1$. Lemma \ref{R17} applied to $C_0$ and $C_1$ gives 
 \begin{equation}
 \label{E108}
 |B\cap C_1|<\binom{d+2}{2}-\binom{d-d+2}{2}-2=\binom{d+2}{2}-3.
 \end{equation}
 However, 
 \begin{equation*}
 B\subseteq B\cap  \psi_d^{-1}(V_d(B)) \subseteq B\cap C_1;
 \end{equation*}
   hence $|B\cap C_1|=|B|=\binom{d+2}{2}-3$ contradicting (\ref{E108}), and this proves (\ref{E107}).  
   
   Now we show that 
   \begin{equation}
   \label{E109}
   \beta_e(B,D)<0.
   \end{equation}
   If (\ref{E109}) is false, then $C_3\neq \emptyset$.  Since $C_2\in\Cc_{\leq e}$ and $C_3\in \Cc_{\leq d-e}$, notice that $C_2\cup C_3\in\Cc_{\leq d}$. From (\ref{E107}), we can apply  Lemma \ref{R17}  to $C_0$ and $C_2\cup C_3$ so  
 \begin{equation}
 \label{E110}
 |B\cap (C_2\cup C_3)|<\binom{d+2}{2}-3.
 \end{equation}
From Remark \ref{R20}.i,
 \begin{equation*}
 B\subseteq \psi_e^{-1}(V_e(D))\cup  \psi_{d-e}^{-1}(W_e(B,D))\subseteq C_2\cup C_3.
 \end{equation*}
 Hence $|B\cap (C_2\cup C_3)|=|B|=\binom{d+2}{2}-3$ contradicting (\ref{E110}), and this proves (\ref{E109}).
 
 Finally, we show that 
   \begin{equation}
   \label{E111}
   \gamma_e(B,D)<\binom{d+2}{2}-\binom{d-e+2}{2}-2.
   \end{equation}
From (\ref{E109}), we get that $C_3=\emptyset$. From (\ref{E107}), $C_0\subseteq C_2$. Thus, since $C_0$ is irreducible, $e\geq d-f$. Applying  Lemma \ref{R17}  to $C_0$ and $C_2$, 
 \begin{equation}
 \label{E112}
 |B\cap C_2|<\binom{d+2}{2}-\binom{d-e+2}{2}-2.
 \end{equation}
 From (\ref{E112}), 
 \begin{equation*}
  \gamma_e(B,D)=|B\cap \psi_e^{-1}(V_e(D))|\leq |B\cap C_2|<\binom{d+2}{2}-\binom{d-e+2}{2}-2,
 \end{equation*}
 and this proves (\ref{E111}). The claim of the lemma is a consequence of (\ref{E109}) and (\ref{E111}). 
     \end{proof}
\begin{lem}
 \label{R27}
 Let $d\in\Nat$, $f\in[0,d-1]$,  $C_0\in\Cc_{\leq d}$ be irreducible and $B\in\Pc(\Real^2)$ be such that $	|B|\leq \binom{d+2}{2}-1$. Then, for all $(e,D)\in I(B,C_0)$, 
 \begin{equation*}
   |C_0\cap U_e(B,D)|\leq 2d^2.
 \end{equation*}
\end{lem} 
\begin{proof}
Fix $p(x,y)\in\Real_d[x,y]$ such that $C_0=\Zc(p(x,y))$. Write $F_1:=V_d(B), F_2:=V_e(D)$ and $F_3:=W_e(B,D)$, and also set $d_1:=d$, $d_2:=e$ and $d_3:=d-e$. Depending on the values of $\alpha_e(D)$ and $\beta_e(B,D)$, there is a subset $I$ of $[1,3]$ such that $U_e(B,D)=\bigcup_{i\in I}\psi_{d_i}^{-1}(F_i)$ with $F_i$ a proper flat in $\Real^{\binom{d_i+2}{2}-1}$ for each $i\in I$ (here we use the assumption $|B|\leq \binom{d+2}{2}-1$ to warranty that $\dim F_1\leq |B|-1<\binom{d+2}{2}-1$). Since $(e,D)\in I(B,C_0)$, we get that 
\begin{equation*}
C_0\nsubseteq U_e(B,D)=\bigcup_{i\in I}\psi_{d_i}^{-1}(F_i).
\end{equation*}
Thus   $C_0\nsubseteq \psi^{-1}_{d_i}(F_i)$  and therefore $\psi_{d_i}(C_0)\nsubseteq F_i$ for each $i\in I$. Since  $F_i$ is a proper flat in $\Real^{\binom{d_i+2}{2}-1}$ and  $\psi_{d_i}(C_0)\nsubseteq F_i$, there exists a hyperplane $H_i$ in  $\Real^{\binom{d_i+2}{2}-1}$ such that $H_i\supseteq F_i$ and $\psi_{d_i}(C_0)\nsubseteq H_i$. For each $i\in I$, set $C_i:=\psi_{d_i}^{-1}(H_i)$ and note that $C_0$ is not a component of $C_i$.  Thus Theorem  \ref{R9} applied to each intersection $|C_0\cap C_i|$ yields 
\begin{align*}
|C_0\cap U_e(B,D)|&\leq \sum_{i\in I} |C_0\cap C_i|\leq d(d_1+d_2+d_3)=2d^2,
\end{align*}
 and this concludes the proof. 
\end{proof}    
     
     \begin{lem}
   \label{R28}
    Let $d\in\Zet$ be such that $d\geq 3$, $f\in[0,d-1]$,  $C_0\in\Cc_{d-f}$ be irreducible and $A\in\Pc(\Real^2)$ be such that $|A\cap C_0|\geq d^52^{\binom{d+2}{2}}$  and $A$ is not contained in an element of $\Cc_{\leq d}$. For all $B_0\in\Pc_{\binom{f+2}{2}}(A\setminus C_0)$  such that $B_0$ is not contained in an element of $\Cc_{\leq f}$, we have that 
    \begin{equation*}
    |\Nc_d(A,B_0,C_0)|\geq \frac{1}{2^{d+3}!}|A\cap C_0|^{\binom{d+2}{2}-3-\binom{f+2}{2}}.
    \end{equation*}
       \end{lem}
        \begin{proof}
      Write $g:=\binom{d+2}{2}-3-\binom{f+2}{2}, A_0:=A\cap C_0$ and $B_0:=\left\{\bb_1,\bb_2,\ldots, \bb_{\binom{f+2}{2}}\right\}$. In this proof,  for any $\obb\in A_0^g$, its $i$-th entry will be denoted by $\bb_{\binom{f+2}{2}+i}$, i.e. $\obb=\left(\bb_{\binom{f+2}{2}+1},\bb_{\binom{f+2}{2}+2},\ldots, \bb_{\binom{d+2}{2}-3}\right)$. For any $\obb\in A_0^g$, write  $B_0(\obb):=B_0$. For all $i\in[1,g]$ and $\obb\in A_0^g$, write
      \begin{align*}
      B_i(\obb)&:=B_0\cup \left\{\bb_{\binom{f+2}{2}+1},\bb_{\binom{f+2}{2}+2},\ldots, \bb_{\binom{f+2}{2}+i}\right\}=\left\{\bb_{1},\bb_{2},\ldots, \bb_{\binom{f+2}{2}+i}\right\}\\
      U_i(\obb)&:=\bigcup_{(e,D)\in I(B_{i-1}(\obb),C_0)}U_e(B_{i-1}(\obb),D).
      \end{align*}
          For all $\obb\in A_0^g$, $i\in[1,g]$ and $(e,D)\in I(B_{i-1}(\obb),C_0)$, we have that $C_0\nsubseteq U_e(B_{i-1}(\obb),D)$. Thus, applying Lemma \ref{R27} to $C_0$ and $U_e(B_{i-1}(\obb),D)$, we get that 
      \begin{equation}
      \label{E114}
      |C_0\cap U_e(B_{i-1}(\obb),D)|\leq 2d^2.
      \end{equation}
             For each $i\in[1,g]$, write $R_i:=\left\{\obb\in A_0^g:\; \bb_{\binom{f+2}{2}+i}\in U_i(\obb)\right\}$. First we show that for all $i\in[1,g]$,
      \begin{equation}
      \label{E115}
      |R_i|\leq d^32^{\binom{d+2}{2}-2}|A_0|^{g-1}.
      \end{equation}
      For each $i\in[1,g], E\in \Pc\left(\left[1,\binom{f+2}{2}+i\right]\right)$ and $\obb\in A_0^g$, set $E(\obb):=\{\bb_j:\;j\in E\}$. Notice that $\Pc\left(B_{i}(\obb)\right)=\left\{E(\obb):\;E\in  \Pc\left(\left[1,\binom{f+2}{2}+i\right]\right)\right\}$. Now, for each $i\in[1,g], E\in \Pc\left(\left[1,\binom{f+2}{2}+i\right]\right)$ and $e\in[1,d-1]$, write 
\begin{equation*}
R_{i,e,E}:=\left\{\obb\in A^g_0:\, \bb_{\binom{f+2}{2}+i}\in U_e(B_{i-1}(\obb),E(\obb))\,\text{and}\,
 (e,E(\obb))\in I(B_{i-1}(\obb),C_0)\right\}
\end{equation*}      
  This definition gives that  for all $i\in[1,g]$, 
 \begin{equation*}
 R_i\subseteq \bigcup_{e=1}^{d-1}\bigcup_{E\in \Pc\left(\left[1,\binom{f+2}{2}+i\right]\right)}R_{i,e,E}
 \end{equation*}
  and thereby
      \begin{equation}
      \label{E116}
       |R_i|\leq \sum_{E\in \Pc\left(\left[1,\binom{f+2}{2}+i\right]\right)}\sum_{e=1}^{d-1}|R_{i,e,E}|.
      \end{equation}
  For each $i\in[1,g], E\in \Pc\left(\left[1,\binom{f+2}{2}+i\right]\right)$ and $e\in[1,d-1]$, notice that if $\obb\in R_{i,e,E}$ then $\bb_{\binom{f+2}{2}+i}\in U_e(B_{i-1}(\obb),E(\obb))\cap A_0\subseteq U_e(B_{i-1}(\obb),E(\obb))\cap C_0$; thus, by (\ref{E114}), $\obb\in R_{i,e,E}$ has at most $2d^2$ possible values in its $i$-th entry and therefore
  \begin{equation}
  \label{E117}
  |R_{i,e,E}|\leq 2d^2|A_0|^{g-1}.
  \end{equation}
             Since $\left|\Pc\left(\left[1,\binom{f+2}{2}+i\right]\right)\times [1,d-1]\right|\leq d2^{\binom{d+2}{2}-3}$, we have that (\ref{E116}) and (\ref{E117}) imply (\ref{E115}).

             Set $T:=\left\{\obb\in A_0^g:\; \bb_{\binom{f+2}{2}+i}\not\in U_i(\obb) \text{ for all }i\in[1,g]\right\}$ so $T=A_0^g\setminus \bigcup_{i=1}^gR_i$. From (\ref{E115}), 
\begin{equation}
\label{E118}
\sum_{i=1}^g|R_i|\leq gd^32^{\binom{d+2}{2}-2}|A_0|^{g-1}\leq d^52^{\binom{d+2}{2}-1}|A_0|^{g-1}. 
\end{equation}  
   Since $|A_0|\geq  d^52^{\binom{d+2}{2}}$ by assumption, we get from (\ref{E118}) that    
             \begin{equation}
             \label{E119}
             |T|=\left|A_0^g\setminus \bigcup_{i=1}^gR_i\right|\geq |A_0|^g-\sum_{i=1}^g|R_i|\geq|A_0|^{g} -d^52^{\binom{d+2}{2}-1}|A_0|^{g-1}\geq \frac{1}{2}|A_0|^{g}.
             \end{equation}
             The second step in the proof is to show that the map
              \begin{equation*}
              \phi:T\longrightarrow \Pc(A),\qquad \phi(\obb)=B_g(\obb)
              \end{equation*}
               satisfies that
      \begin{equation}
      \label{E120}
      \phi(T)\subseteq \Nc_d(A,B_0,C_0).
\end{equation} 
We claim that for all $\obb\in T$, $i\in[0,g]$ and $(e,D)\in I(B_i(\obb),C_0)$, we have that 
\begin{equation}
\label{E121}
\max\left\{\tau_e(B_i(\obb),D),\mu_e(B_i(\obb),D)\right\}<\binom{d+2}{2}.
\end{equation}    
We show (\ref{E121}) by induction on $i\in[0,g]$. For $i=0$, we apply Lemma \ref{R25} to $B_0=B_0(\obb)$. Assume that (\ref{E121}) holds for $i-1\geq 0$ and we show it for $i$. In so far as $\obb\in T$, note that $\bb_i\in A\setminus U_i(\obb)$. By induction,
\begin{equation*}
\max\left\{\tau_e(B_{i-1}(\obb),D),\mu_e(B_{i-1}(\obb),D)\right\}<\binom{d+2}{2}
\end{equation*}
for all  $(e,D)\in I\left(B_{i-1}(\obb),C_0\right)$ so $B_{i-1}(\obb)$ and $\bb_i\in A\setminus U_i(\obb)$ satisfy the assumptions of Lemma \ref{R22}, and hence this lemma implies that $B_{i-1}(\obb)\cup\{\bb_i\}=B_{i}(\obb)$ satisfies (\ref{E121}) completing the induction. In particular, (\ref{E121}) is true for $i=g$ so, for all $\obb\in T$ and $(e,D)\in I\left(B_{g}(\obb),C_0\right)$,
\begin{equation}
\label{E122}
\max\left\{\tau_e(B_g(\obb),D),\mu_e(B_g(\obb),D)\right\}<\binom{d+2}{2}.
\end{equation}
Now, for all $(e,D)\in ([1,d-1]\times \Pc(B_{g}(\obb)))\setminus I\left(B_{g}(\obb),C_0\right)$ such that $\alpha_e(D)\geq 0$, Lemma \ref{R26} implies that 
 \begin{align}
 \label{E123}
 \beta_e(B,D)&<0\\
  \label{E124}
  \gamma_e(B,D)&<\binom{d+2}{2}-\binom{d-e+2}{2}-2.
 \end{align}
  For all $\obb\in T$ and $i\in[1,g]$, we get that $\bb_{\binom{f+2}{2}+i}\not\in U_i(\obb)$; in particular, $\bb_{\binom{f+2}{2}+i}\not\in\\\psi^{-1}_d(V_d(B_{i-1}(\obb)))$. Then the flats $V_d(B_0(\obb))\subsetneq V_d(B_1(\obb))\subsetneq\ldots \subsetneq V_d(B_g(\obb))$ are contained properly so 
  \begin{equation}
 \label{EE125}
\dim V_d(B_g(\obb))\geq g+\dim V_d(B_0(\obb)).
 \end{equation}
  Since $B_0=B_0(\obb)$ is not contained in an element of $\Cc_{\leq f}$, we get that $\psi_f(B_0(\obb))$ is not contained in a hyperplane of $\Real^{\binom{f+2}{2}-1}$ so $\dim \psi_f(B_0(\obb))=\binom{f+2}{2}-1$. From Remark \ref{R12}.iii, we get that $\dim \psi_d(B_0(\obb))\geq \dim \psi_f(B_0(\obb))$ so 
  \begin{equation*}
  \dim V_d(B_0(\obb))=\dim \psi_d(B_0(\obb))\geq \dim \psi_f(B_0(\obb))=\binom{f+2}{2}-1,
  \end{equation*}
  and then (\ref{EE125}) gives   
 \begin{equation}
 \label{E125}
\dim V_d(B_g(\obb))\geq \binom{d+2}{2}-4.
 \end{equation}
 On the other hand, 
 \begin{equation*}
 \dim V_d(B_g(\obb))=\dim \Fl(\psi_d(B_g(\obb)))\leq |\psi_d(B_g(\obb))|-1=|B_g(\obb)|-1=\binom{d+2}{2}-4
 \end{equation*}
 so (\ref{E125}) leads to 
  \begin{equation}
 \label{E126}
 \dim \psi_d(B_g(\obb))=\dim V_d(B_g(\obb))=\binom{d+2}{2}-4.
 \end{equation}
 We can apply Lemma \ref{R23} with $B=B_g(\obb)$, $B_0$ and $C_0$ because its assumptions are satisfied (indeed: i) holds by (\ref{E126}); ii) and iii) hold by the construction of $B_g(\obb)$; iv) holds by (\ref{E122}); v) holds by (\ref{E123}) and (\ref{E124})).   Hence Lemma \ref{R23} implies that $B_g(\obb)\in\Nc_d(A, B_0, C_0)$ and this shows (\ref{E120}).
 
Finally, for each $B\in \Nc_d(A, B_0, C_0)$, $|\phi^{-1}(B)|\leq g!$  which is the number of permutations of the elements in $B\setminus B_0$. Therefore (\ref{E119}) and (\ref{E120}) give
 \begin{equation}
 \label{E127}
 |\Nc_d(A, B_0, C_0)|\geq \frac{1}{g!}|T|\geq  \frac{1}{g!2}|A_0|^g.
\end{equation} 
Since $g\leq \binom{d+2}{2}-3<2^{d+2}$, we have that $g!2\leq 2^{d+3}!$  and hence (\ref{E127}) completes the proof. 
      \end{proof}
        \section{From curves to lines}
      In this section we use the elements of the families $\Nc_d(A)$ to generate flats that will be used to construct hyperprojections  which are used to transform the problem of finding curves of degree $d$ with few points   into a problem of finding ordinary lines  which avoid a finite set. 
      
     Let $d\in\Nat, e\in[1,d-1]$ and  $B\in\Pc(\Real^2)$. Write
\begin{align*}
\Cc_{d,e}(B)&:=\left\{C\in\Cc_e:\; |C\cap B|=\binom{d+2}{2}-\binom{d-e+2}{2}-1\right\}\\
\Cc_d(B)&:=\bigcup_{f=1}^{d-1}\Cc_{d,f}(B).
\end{align*}     
\begin{lem}
\label{R29}
 Let $d\in\Nat$ and $e\in[1,d-1]$. For any $B\in\Nc_d(\Real^2)$, 
 \begin{equation*}
 |\Cc_{d,e}(B)|<2^{2^{d+2}}.
 \end{equation*}
\end{lem} 
   \begin{proof}
   Since $B\in \Nc_d(\Real^2)$, we have that  for any $f\in [1,d-1]$ and  $C\in\Cc_f$,
   \begin{equation*}
   |C\cap B|<\binom{d+2}{2}-\binom{d-f+2}{2}\leq \binom{d+2}{2}-3=|B|;
   \end{equation*}
   in particular, this means that there is no element  $C\in\Cc_{\leq e}$ which contains $B$. Thus, from Remark \ref{R12}.ii, no hyperplane in $\Real^{\binom{e+2}{2}-1}$ can contain $\psi_e(B)$ and therefore
   \begin{equation}
   \label{E128}
   \dim \psi_e(B)=\binom{e+2}{2}-1.
   \end{equation}
  Set the map
  \begin{equation*}
  \phi:\Cc_{d,e}(B)\longrightarrow \Pc_{\binom{d+2}{2}-\binom{d-e+2}{2}-1}(B),\qquad \phi(C)=C\cap B
  \end{equation*}
  We will show that $\phi$ is injective. Take $C_1,C_2\in\Cc_{d,e}(B)$ such that $ C_1\cap B=C_2\cap B$.   Write $F:=\Fl(\psi_e(C_1\cap B))=\Fl(\psi_e(C_2\cap B))$. Because  $C_1,C_2\in\Cc_{\leq e}$, we have that $\Fl(\psi_e(C_1))$ and $\Fl(\psi_e(C_2))$ are contained in hyperplanes by Remark \ref{R12}.ii. For $i\in\{1,2\}$,
  \begin{equation}
  \label{E129}
  F\subseteq \Fl(\psi_e(C_i\cap B))\subseteq \Fl(\psi_e(C_i)) 
  \end{equation}
 so $F$ is a proper flat in $\Real^{\binom{e+2}{2}-1}$. We claim that $F$ is a hyperplane. If $F$ is not a hyperplane, then (\ref{E128}) implies that we can choose a  hyperplane $H$ in $\Real^{\binom{e+2}{2}-1}$ such that $H\supseteq F$ and 
\begin{equation}
\label{E130}
|H\cap \psi_e(B)|>|F\cap \psi_e(B)|.
\end{equation}  
  However,
  \begin{align*}
  \binom{d+2}{2}-\binom{d-e+2}{2}-1&\geq |\psi_e^{-1}(H)\cap B|&\Big(\text{since $B\in\Nc_d(\Real^2)$}\Big)\\
  &=|H\cap \psi_e(B)| \\
  &>|F\cap \psi_e(B)|&\Big(\text{by  (\ref{E130}})\Big)\\
  &\geq |C_1\cap B|\\
  &=\binom{d+2}{2}-\binom{d-e+2}{2}-1,&\Big(\text{since $C_1\in\Cc_{d,e}(B)$}\Big)
    \end{align*}
  and this contradiction shows that $F$ is a hyperplane. Hence, since $F$ is a hyperplane, (\ref{E129}) leads to 
  \begin{equation*}
  \Fl(\psi_e(C_1))=F= \Fl(\psi_e(C_2)),
\end{equation*}   
 and then Remark \ref{R12}.ii yields that $C_1=C_2$ implying that $\phi$ is injective. In so far as $\phi$ is injective,
 \begin{equation}
 \label{E131}
  |\Cc_{d,e}(B)|\leq \left|\Pc_{\binom{d+2}{2}-\binom{d-e+2}{2}-1}(B)\right|\leq |\Pc(B)|=2^{|B|}=2^{\binom{d+2}{2}-3}.
 \end{equation}
 Using that $\binom{d+2}{2}-3<2^{d+2}$, the claim follows from (\ref{E131}).
\end{proof}    
Let $d\in\Nat$. As usual, $\Pro^d:=(\Real^{d+1}\setminus\{\ori\})/\sim$ with $\yb\sim \zb$ if there is $r\in\Real$ such that $\yb=r\cdot \zb$. We will use the embedding 
\begin{equation*}
\Real^d\longrightarrow \Pro^d,\qquad (z_1,z_2,\ldots,z_d)\mapsto [1:z_1:z_2:\ldots:z_d]
\end{equation*}
 so that $\Real^d$ can be seen as a subset of $\Pro^d$. For any flat $F$ in $\Real^d$ defined by a family of linear equations $\left\{r_{0,i}+\sum_{j=1}^dr_{j,i}x_j\right\}_{i\in I}$ (i.e. $F=\bigcap_{i\in I}\Zc\left(r_{0,i}+\sum_{j=1}^dr_{j,i}x_j\right)$), its \emph{homogenization} will be the subset of elements $[z_0:z_1:\ldots:z_d]\in \Pro^d$ such that  $\sum_{j=0}^dr_{j,i}z_j=0$ for all $i\in I$, and we will denote it by $F^h$. The next standard facts about homogenization can be found in \cite[Sec.I.2]{Ha}.
 
 \begin{rem}
\label{R30}
Let $d\in \Nat$, $e\in[1,d]$ and $F$ be an $e$-dimensional flat in $\Real^d$.
\begin{enumerate}
\item[i)]Then $F^h$ is an $e$-dimensional linear variety in $\Pro^d$, and $F^h\cong \Pro^e$.
\item[ii)]Let $f\in[1,d]$ and  $G$ be an $f$-dimensional flat in $\Real^d$. If $e+f-d\geq 0$, then $F^h\cap G^h\neq\emptyset$ and  $F^h\cap G^h$ is a $g$-dimensional linear variety in $\Pro^d$ with $g\geq e+f-d$.
\end{enumerate}
\end{rem}

 Let $d\in\Nat, e\in [0,d-1]$ and $F$ be an $e$-dimensional flat in $\Real^d$. Take $G$ a $d-e-1$-dimensional flat in $\Real^d$ such that $F\cap G=\emptyset$. For any $\zb\in \Real^d\setminus F$, notice that $\dim \Fl(F\cup\{\zb\})=e+1$ and  $\dim \Fl(F\cup\{\zb\})\cap G\leq 0$ since $F\cap G=\emptyset$. On the other hand, Remark \ref{R30}.ii implies that  $\Fl(F\cup\{\zb\})^h\cap G^h$ is a linear variety of dimension at least $0$. Therefore $\Fl(F\cup\{\zb\})^h\cap G^h$ is exactly a point. Identifying $G^h$ with $\Pro^{d-e-1}$, we define the \emph{hyperprojection centered in $F$} as the map 
\begin{equation*}
\pi_F:\Real^d\setminus F\longrightarrow \Pro^{d-e-1},\qquad \{\pi_{F}(\zb)\}=\Fl(F\cup\{\zb\})^h\cap \Pro^{d-e-1}.
\end{equation*}
An easy consequence of the the definition of $\pi_F$ is the following remark.
\begin{rem}
\label{R31}
 Let $d\in\Nat, e\in [0,d-1]$ and $F$ be an $e$-dimensional flat in $\Real^d$. For any $f\in [e+1,d]$, there is an bijective relation between the $f$-flats  containing $F$ and the $f-e-1$-dimensional linear varieties in $\Pro^{d-e-1}$ given by $H\mapsto \pi_F(H\setminus F)$.
\end{rem}
Let  $B\in\Nc_d(\Real^2)$. Recall that $V_d(A)=\Fl(\psi_d(A))$ for any  $A\in\Pc(\Real^2)$. Write 
\begin{align*}
D_d(B)&:=\psi^{-1}_d(V_d(B))\\
E_d(B)&:=D_d(B)\cup \bigcup_{C\in\Cc_d(B)}\psi^{-1}_d(V_d(C\cup B)),
\end{align*}
 and the map
 \begin{equation*}
 \varphi_{d,B}:\Real^2\setminus D_d(B)\longrightarrow  \Pro^{2},\qquad \varphi_{d,B}(\ab)=\pi_{V_d(B)}(\psi_d(\ab));
 \end{equation*}
  in other words, $\varphi_{d,B}=\pi_{V_d(B)}\circ \psi_d|_{\Real^2\setminus D_d(B)}$. The main properties of $\varphi_{d,B}$ are proven in the next lemma. 
  
\begin{lem}
\label{R32}
Let $d\in\Nat$ and $B\in\Nc_d(\Real^2)$. 
\begin{enumerate}
\item[i)]For all $e\in[1,d-1]$ and $C\in \Cc_{d,e}(B)$, we have that $|\varphi_{d,B}(C\setminus D_d(B))|=1$.
\item[ii)]For all $\ab\in\Real^2\setminus E_d(B)$, we get that $\varphi_{d,B}(\ab)\not\in\varphi_{d,B}(E_d(B)\setminus D_d(B))$.
\item[iii)]For all $\ab\in\Real^2\setminus E_d(B)$, we have  that $\left|\varphi_{d,B}^{-1}\left(\varphi_{d,B}(\ab)\right)\right|\leq d^2-|D_d(B)|$.
\end{enumerate}
\end{lem}
\begin{proof}
Since $B\in \Nc_d(\Real^2)$, notice that for all $e\in [1,d-1]$ and $C\in \Cc_{d,e}(B)$,
\begin{equation}
\label{E132}
\dim\psi_{d-e}(B\setminus C)=\binom{d-e+2}{2}-3;
\end{equation}
in particular, (\ref{E132}) implies there is a hyperplane $H$ in $\Real^{\binom{d-e+2}{2}-1}$ such that $\psi_{d-e}(B\setminus C)\subseteq H$ and therefore $B\setminus C\subseteq \psi_{d-e}^{-1}(H)$. In so far as 
$\psi_{d-e}^{-1}(H)\in \Cc_{\leq d-e}$ and $C\in\Cc_e$, we get that $\psi_{d-e}^{-1}(H)\cup C\in \Cc_{\leq d}$ by Remark \ref{R10}.i. Thus, insomuch as 
\begin{equation*}
B\cup C=(B\setminus C)\cup C\subseteq \psi_{d-e}^{-1}(H)\cup C, 
\end{equation*}
Remark \ref{R12}.ii implies that $\psi_d(B\cup C)$ is contained in hyperplane $K$ of $\Real^{\binom{d+2}{2}-1}$ such that $\psi_d^{-1}(K)=\psi_{d-e}^{-1}(H)\cup C$; in particular,
\begin{equation}
\label{E133}
\dim\psi_{d}(B\cup C)\leq \binom{d+2}{2}-2.
\end{equation}

First we prove i). To show i), it suffices to prove that
\begin{equation}
\label{E134}
\dim \psi_d(B\cup C)=\binom{d+2}{2}-3
\end{equation} 
because $\pi_{V_d(B)}$ projects  $\binom{d+2}{2}-3$-flats which contain $V_d(B)=\Fl(\psi_d(B))$ into points of $\Pro^2$. From (\ref{E132}),  $V_{d-e}(B\setminus C)$ is a proper flat in $\Real^{\binom{d-e+2}{2}-1}$,  and from  (\ref{E133}), $V_{d}(B\cup C)$ is a proper flat in $\Real^{\binom{d+2}{2}-1}$; therefore
\begin{equation*}
\Real^2\setminus(\psi_{d-e}^{-1}(V_{d-e}(B\setminus C))\cup \psi_{d}^{-1}(V_{d}(B\cup C)))\neq \emptyset.
\end{equation*}
 Fix $\ab\in \Real^2\setminus(\psi_{d-e}^{-1}(V_{d-e}(B\setminus C))\cup \psi_{d}^{-1}(V_{d}(B\cup C)))$. Since $\ab\not\in \psi_{d-e}^{-1}(V_{d-e}(B\setminus C))$, we get from (\ref{E132}) that 
  \begin{equation}
\label{E135}
\dim\psi_{d-e}( (B\cup\{\ab\})\setminus C)=\dim\psi_{d-e}(\{\ab\}\cup (B\setminus C))=\binom{d-e+2}{2}-2.
\end{equation}
Now we apply Lemma \ref{R18}.ii to $B\cup\{\ab\}$ so (\ref{E135}) yields
\begin{equation}
\label{E136}
 \dim\psi_{d}((B\cup\{\ab\})\cup C)=\binom{d+2}{2}-2.
\end{equation}
In so far as  $\ab\not\in \psi_{d}^{-1}(V_{d}(B\cup C))$, (\ref{E136}) leads to (\ref{E134}).

We prove ii) by contradiction. Assume that there is $\ab\in\Real^2\setminus E_d(B)$ such that $\varphi_{d,B}(\ab)\in\varphi_{d,B}(E_d(B)\setminus D_d(B)))$. Therefore there are $e\in[1,d-1]$ and $C\in\Cc_{d,e}(B)$ such that $\varphi_{d,B}(\ab)=\varphi_{d,B}(C\setminus D_d(B)))$. This equality means that $\psi_d(\ab)$ is in the $\binom{d+2}{2}-3$-flat $V_d(B\cup C)=\Fl(\psi_d(B\cup C))$ so 
\begin{equation*}
\ab\in \psi^{-1}_d(V_d(B\cup C))\subseteq E_d(B),
\end{equation*}
 which contradicts the assumption. 
 
 Finally, we show iii) by contradiction. Assume that there is $\ab\in\Real^2\setminus E_d(B)$ such that 
\begin{equation}
\label{E137}
\left|\varphi_{d,B}^{-1}\left(\varphi_{d,B}(\ab)\right)\right|+|D_d(B)|> d^2.
\end{equation} 
 Since $B\in \Nc_d(\Real^2)$, we have that $\dim \psi_d(B)=\binom{d+2}{2}-4$. Now, in so far as $\ab\not\in D_d(B)=\psi^{-1}_d(\Fl(\psi_d(B)))$, we get that 
\begin{equation}
\label{E138}
\dim \psi_d(\{\ab\}\cup B)=\binom{d+2}{2}-3;
\end{equation} 
 thus $V_d(\{\ab\}\cup B)=\Fl(\psi_d(\{\ab\}\cup B))$ is a $\binom{d+2}{2}-3$-flat projected to $\varphi_{d,B}(\ab)$ by $\pi_{V_d(B)}$. Since $V_d(\{\ab\}\cup B)$ is a $\binom{d+2}{2}-3$-flat, there are $H_1$ and $H_2$ distinct hyperplanes in $\Real^{\binom{d+2}{2}-1}$ such that $H_1\cap H_2=V_d(\{\ab\}\cup B)$; write $C_1:=\psi_d^{-1}(H_1)$ and $C_2:=\psi_d^{-1}(H_2)$ so that 
 \begin{equation}
 \label{E139}
 C_1\cap C_2=\psi_d^{-1}(H_1\cap H_2)=\psi_d^{-1}(V_d(\{\ab\}\cup B)),
 \end{equation}
  and therefore 
  \begin{equation}
   \label{E140}
   |C_1\cap C_2|=|\psi_d^{-1}(V_d(\{\ab\}\cup B))|\geq \left|\varphi_{d,B}^{-1}\left(\varphi_{d,B}(\ab)\right)\right|+|D_d(B)|.
  \end{equation}
  On the other hand,  Remark \ref{R12}.ii leads to $C_1,C_2\in\Cc_{\leq d}$.  We claim that $C_1$ and $C_2$ share an irreducible component. If this claim were false, then  Theorem \ref{R9} would give $ |C_1\cap C_2|\leq d^2$;  however, this contradicts the inequality $|C_1\cap C_2|>d^2$ which is a consequence of  (\ref{E137}) and (\ref{E140}). Therefore there is an irreducible curve $C_0\in\Cc_e$  for some $e\in[1,d-1]$ such that $C_0\subseteq C_1\cap C_2$. From (\ref{E139}), $C_0\subseteq \psi_d^{-1}(V_d(\{\ab\}\cup B))$ so $\psi_d(C_0)\subseteq V_d(\{\ab\}\cup B)=\Fl(\psi_d(\{\ab\}\cup B))$; this and (\ref{E138}) yield
  \begin{equation}
  \label{E142}
 \dim \psi_d(\{\ab\}\cup B\cup C_0)=\dim \psi_d(\{\ab\}\cup B)=\binom{d+2}{2}-3.
  \end{equation}
  Since $B\in\Nc_d(\Real^2)$, notice that
  \begin{equation*}
  |B\cap C_0|\leq \binom{d+2}{2}-\binom{d-e+2}{2}-1.
  \end{equation*}
 Thus  we have two cases.
\begin{enumerate}
\item[$\bullet$] If $|B\cap C_0|= \binom{d+2}{2}-\binom{d-e+2}{2}-1$, then $C_0\in\Cc_{d,e}(B)$. From (\ref{E134}), $\dim \psi_d(B\cup C_0)=\binom{d+2}{2}-3$; however, since $\ab\not\in \psi^{-1}_d(V_d(B\cup C_0))$ (because  $\ab\in\Real^2\setminus E_d(B)$), we get that 
  \begin{equation*}
\dim \psi_d(\{\ab\}\cup B\cup C_0)>\dim \psi_d(B\cup C_0)=\binom{d+2}{2}-3
  \end{equation*}
   contradicting (\ref{E142}).

\item[$\bullet$]If $|B\cap C_0|< \binom{d+2}{2}-\binom{d-e+2}{2}-1$, then (because $B\in\Nc_d(\Real^2)$) we have that  
\begin{equation}
\label{E143}
\dim\psi_{d-e}(B\setminus C_0)>\binom{d-e+2}{2}-3.
\end{equation}   
Using (\ref{E143}) and Lemma \ref{R18}, we conclude that $\dim\psi_{d}(B\cup C_0)>\binom{d+2}{2}-3$ but (\ref{E142}) gives
\begin{equation*}
\binom{d+2}{2}-3=\dim \psi_d(\{\ab\}\cup B\cup C_0)\geq \dim \psi_d(B\cup C_0)>\binom{d+2}{2}-3,
\end{equation*}
which is impossible. 
\end{enumerate}  
  In any case, we reached a contradiction and this concludes the proof of iii).
 \end{proof}
 Let $d,n\in\Nat, A\in\Pc(\Real^2)$ be finite  and $B\in\Nc_d(A)$. Set
 \begin{align*}
 \Oc_{d,n}(A,B)&:=\{C\in \Oc_{d,n}(A):\;B\subseteq C\}\\
 \delta_{d}(A,B)&:=\max_{\ab\in A\setminus E_d(B)}\left|\varphi_{d,B}^{-1}\left(\varphi_{d,B}(\ab)\right)\right|.
 \end{align*}
 From Lemma \ref{R32}, $\delta_{d}(A,B)$ exists and
\begin{equation*}
\delta_{d}(A,B)\leq d^2-|D_d(B)|\leq d^2-|B|=\frac{d^2-3d+4}{2}.
\end{equation*} 
 For any $m\in\Nat$, let $c_6=c_6(m)$ and $c_7=c_7(m)$ be as in Lemma \ref{R19} and set 
  \begin{align*}
 c_8(d)&:=d^2\cdot \max_{1\leq m\leq d2^{2^{d+2}}}c_6(m)\\
  c_9(d)&:=d^2\cdot \max_{1\leq m\leq d2^{2^{d+2}}}c_7(m).
 \end{align*}
 Then next two lemmas are the main results of this section. 
\begin{lem}
\label{R33}
  Let $d\in\Nat$ and $A\in\Pc(\Real^2)$ be finite, $d$-regular and  such that $|A|\geq d^22^3\max\{c_8(d),c_9(d),1\}$. Write $n:=2\delta_{d}(A,B)+|D_d(B)|$.  Then, for any   $B\in\Nc_d(A,\emptyset,\Real^2)$,
  \begin{equation*}
  \left|\Oc_{d,n}(A,B)\right|\geq \frac{1}{d^{d+2}2^4}|A|.
  \end{equation*}
\end{lem}

 \begin{proof}
Since $A$ is finite, we can apply a linear automorphism in $\Pro^2$ to assume that $\varphi_{d,B}(A\setminus D_d(B))\cap  (\Pro^2\setminus \Real^2)=\emptyset$; thus we assume from now on that  $\varphi_{d,B}(A\setminus D_d(B))\subseteq \Real^2$.   Write  $S:=\varphi_{d,B}(A\setminus E_d(B)), T:=\varphi_{d,B}(E_d(B)\setminus D_d(B))$ and $\Lc:=\{L\in\Oc_2(S):\;L\cap T=\emptyset\}$. From Lemma \ref{R32}.ii, we have that $S\setminus T=S$.  From Lemma \ref{R29}, we have that $ |\Cc_{d,e}(B)|<2^{2^{d+2}}$ for all $e\in[1,d-1]$ so $ |\Cc_{d}(B)|<d2^{2^{d+2}}$. Thus Lemma \ref{R32}.i and the previous inequality lead to
\begin{equation}
\label{E144}
 |T|\leq |\Cc_{d}(B)|<d2^{2^{d+2}}.
\end{equation} 
 For each $C\in\Cc_{d}(B)$, fix a hyperplane $H_C$ in $\Real^{\binom{d+2}{2}-1}$ such that $H_C\supseteq V_d(C\cup B)$; since $\psi_d^{-1}(H_C)\in\Cc_{\leq d}$ and $A$ is $d$-regular,
\begin{equation}
\label{E145}
\left|A\cap \psi_d^{-1}(H_C)\right|\leq  \frac{1}{2^{2^{3d+8}}}|A|.
\end{equation}
From (\ref{E144}) and (\ref{E145}),
\begin{equation*}
\sum_{C\in\Cc_d(B)}\left|A\cap \psi_d^{-1}(H_C)\right|\leq \frac{d2^{2^{d+2}}}{2^{2^{3d+8}}}|A|\leq \frac{|A|}{2},
\end{equation*}
and therefore
\begin{equation}
\label{E146}
|A\setminus E_d(B)|=|A|-|A\cap E_d(B)|\geq |A|-\sum_{C\in\Cc_d(B)}\left|A\cap \psi_d^{-1}(H_C)\right|\geq \frac{|A|}{2}.
\end{equation}
Since $\delta_d(A,B)\leq d^2$ by Lemma \ref{R32}.iii, we get from (\ref{E146}) that 
\begin{equation}
\label{E147}
|S|\geq\frac{1}{\delta_d(A,B)} |A\setminus E_d(B)|\geq \frac{1}{d^2}|A\setminus E_d(B)|\geq \frac{1}{2d^2}|A|.
\end{equation}
On the one hand, $S$ is not collinear because if $S$ is contained in a line $L$, then $A\setminus E_d(B) $ is contained in $\varphi_{d,B}^{-1}(L)$ which is in $\Cc_{\leq d}$, and  (\ref{E146}) would contradict the $d$-regularity of $A$. On the other hand, (\ref{E147}) implies $|S|>c_8(d)$. Thus we can apply Lemma \ref{R19} to $S$ and $T$, and we obtain that 
\begin{equation}
\label{E148}
|\Lc|\geq \frac{1}{2}|S|-c_9(d).
\end{equation}
From (\ref{E147}) and (\ref{E148}),
\begin{equation}
\label{E149}
|\Lc|\geq \frac{1}{d^22^4}|A|.
\end{equation}
Denote by $\Hc$ the family of hyperplanes in $\Real^{\binom{d+2}{2}-1}$ generated by $\psi_d(A)$ and define the maps
\begin{align*}
\eta_1:\Lc\longrightarrow \Hc,&\qquad \eta_1(L)=\pi_{V_d(B)}^{-1}(L)\\
\eta_2:\eta_1(\Lc)\longrightarrow \Cc_{\leq d},&\qquad \eta_2(\eta_1(L))=\varphi_{d,B}^{-1}(L).
\end{align*}
We will show that 
\begin{equation}
\label{E150}
\eta_2(\eta_1(\Lc))\subseteq \Oc_{d,n}(A,B).
\end{equation}
For each $L\in\Lc$, $L$ is a line generated by elements of $S=\varphi_{d,B}(A\setminus D_d(B))$. Hence $\eta_1(L)$ is a hyperplane generated by elements of $\psi_d(A)$, and then, by Lemma \ref{R13}, $\eta_2(\eta_1(L))=\varphi_{d,B}^{-1}(L)=\psi_d^{-1}(\eta_1(L))$ is an element of $\Cc_{\leq d}$ determined by $A$ (i.e. $\eta_2(\eta_1(L))\in\Dc_d(A)$). For any $L\in\Lc$, we have that $L\cap S=\{\varphi_{d,B}(\ab_1),\varphi_{d,B}(\ab_2)\}$  for some $\ab_1,\ab_2\in A\setminus E_d(B)$. In so far as  $\eta_2(\eta_1(L))=\varphi_{d,B}^{-1}(L)$, Lemma \ref{R32} leads to 
\begin{equation}
\label{E151}
 \eta_2(\eta_1(L))\cap A=\varphi_{d,B}^{-1}(L)\cap A=\left(\varphi_{d,B}^{-1}(\varphi_{d,B}(\ab_1))\cup \varphi_{d,B}^{-1}(\varphi_{d,B}(\ab_2))\cup D_d(B)\right)\cap A.
 \end{equation} 
 We conclude from (\ref{E151}) that 
\begin{equation*}
|\eta_2(\eta_1(L))\cap A|\leq |\varphi_{d,B}^{-1}(\varphi_{d,B}(\ab_1))|+|\varphi_{d,B}^{-1}(\varphi_{d,B}(\ab_2))|+|D_d(B)|\leq n,
\end{equation*}
and therefore $\eta_2(\eta_1(L))\in\Oc_{d,n}(A,B)$ proving (\ref{E150}). From Remark \ref{R31}, $\eta_1$ is injective.  From Lemma \ref{R11}, for each $C\in\Cc_{\leq d}$, there are at most $d^d$ hyperplanes in $\Real^{\binom{d+2}{2}-1}$ such that $C=\psi_d^{-1}(H)$ so $|\eta_2^{-1}(C)|\leq d^d$ for all $C\in \Cc_{\leq d}$. The previous two statements  give $|\eta_2(\eta_1(\Lc))|\geq \frac{1}{d^d}|\Lc|$, and then (\ref{E150}) yields that
\begin{equation}
\label{E152}
\left|\Oc_{d,n}(A,B)\right|\geq |\eta_2(\eta_1(\Lc))|\geq \frac{1}{d^d}|\Lc|.
\end{equation}
Then   the lemma follows from (\ref{E149}) and (\ref{E152}).
 \end{proof}
 
 \begin{lem}
\label{R34}
  Let $d\in\Nat$, $f\in[0,d-1]$,  $C_0\in\Cc_{d-f}$ be irreducible,  $A\in\Pc(\Real^2)$ be finite such that 
$|A\cap C_0|\geq d^{d+8}2^{2^{d+8}}\max\{c_8(d),c_9(d),1\}$,   and $B_0\in\Pc_{\binom{f+2}{2}}(A\setminus C_0)$ be such that $B_0$ is not contained in an element of $\Cc_{\leq f}$.  Write $n:=2\delta_{d}(A,B)+|D_d(B)|$.
\begin{enumerate}
\item[i)]For any $B\in\Nc_d(A,B_0,C_0)$, we get that $\left|\Oc_{d,n}(A,B)\right|\geq \frac{1}{2^3d^{d+2}}|A\cap C_0|$.
\item[ii)]Assume that $|A\setminus C_0|\leq \frac{1}{2^2d^4}|A\cap C_0|$. Then, for any $B\in\Nc_d(A,B_0,C_0)$, we get that $ \left|\Oc_{d,d^2}(A,B)\right|\geq \frac{1}{2^4d^{d+8}}|A\cap C_0|^2$.
\end{enumerate}
\end{lem}
\begin{proof}
 Write  $S:=\varphi_{d,B}(A\setminus E_d(B)), S_0:=\varphi_{d,B}((A\cap C_0)\setminus E_d(B))$ and $T:=\varphi_{d,B}(E_d(B)\setminus D_d(B))$. Denote by $\Hc$ the family of hyperplanes in $\Real^{\binom{d+2}{2}-1}$ generated by $\psi_d(A)$, denote by $\Lc$ the family of lines generated by $\varphi_{d,B}(A\setminus D_d(B))$ and define the maps
\begin{align*}
\eta_1:\Lc\longrightarrow \Hc,&\qquad \eta_1(L)=\pi_{V_d(B)}^{-1}(L)\\
\eta_2:\eta_1(\Lc)\longrightarrow \Cc_{\leq d},&\qquad \eta_2(\eta_1(L))=\psi^{-1}_{d}(\eta_1(L)).
\end{align*}
 As in the first part of Lemma \ref{R33}, we may assume that $S$ is contained in $\Real^2$ and we have that
\begin{equation}
\label{E153}
 |T|\leq |\Cc_{d}(B)|<d2^{2^{d+2}}.
\end{equation} 

 The next step is to show that  for any hyperplane $H$ in $\Real^{\binom{d+2}{2}-1}$containing $V_d(B)$, 
\begin{equation}
\label{E154}
|H\cap \psi_d(C_0)|\leq d^2.
\end{equation}
Indeed, write  $C:=\psi_d^{-1}(H)$. Note that 
\begin{equation*}
B\subseteq \psi^{-1}_d(V_d(B))\subseteq \psi^{-1}_d(H)=C
\end{equation*}
 so $|B\cap C|=|B|=\binom{d+2}{2}-3$.  Thus $C_0$ is not a component of $C$ because otherwise Lemma \ref{R17} applied to $C_0$ and $C$ leads to $|B\cap C|<\binom{d+2}{2}-3$.  Thereby Theorem \ref{R9} applied to $C$ and $C_0$ gives 
\begin{align*}
|\psi_d(C_0)\cap H|&=|C_0\cap C|\leq d^2,
\end{align*} 
which shows (\ref{E154}). Remark \ref{R31}  implies that for any line $L\in\Real^2$, there is a hyperplane $H$ in $\Real^{\binom{d+2}{2}-1}$containing $V_d(B)$ such that $\pi_{V_d(B)}(H)=L$. Then (\ref{E154}) yields that for any line $L$ in $\Real^2$, 
\begin{equation}
\label{E156}
|L\cap \varphi_{d,B}(C_0\setminus D_d(B))|\leq d^2.
\end{equation}

The next step is to prove that 
\begin{equation}
\label{E157}
\left|S_0\right|\geq \frac{1}{d^2}\left(|A\cap C_0|-d^42^{2^{d+3}}\right)\geq \frac{1}{2d^2}|A\cap C_0|.
\end{equation}
Denote by $\Lc_T$ the family of lines generated by $T$. From (\ref{E153}),
\begin{equation}
\label{E158}
|\Lc_T|\leq \binom{|T|}{2}<d^22^{2^{d+3}}.
\end{equation} 
For any $L\in\Lc_T$, the hyperplane $\pi_{V_d(B)}^{-1}(L)$ contains $V_d(B)$ so (\ref{E154}) leads to
\begin{equation}
\label{E159}
|\pi_{V_d(B)}^{-1}(L)\cap \psi_d(C_0)|\leq d^2.
\end{equation}
From (\ref{E158}) and (\ref{E159}),
\begin{equation}
\label{E160}
\left|\psi_d(A\cap C_0)\cap \bigcup_{L\in\Lc_T}\pi_{V_d(B)}^{-1}(L)\right|\leq \sum_{L\in\Lc_T}|\psi_d(C_0)\cap \pi_{V_d(B)}^{-1}(L)|<d^42^{2^{d+3}}.
\end{equation}
Since $T$ is contained $\bigcup_{L\in\Lc_T}L$, we have that $E_d(B)$ is contained in \\$\psi_d^{-1}\left(\bigcup_{L\in\Lc_T}\pi_{V_d(B)}^{-1}(L)\right)$. Thus (\ref{E160}) leads to 
\begin{equation}
\label{E161}
|(A\cap C_0)\setminus E_d(B)|\geq \left|(A\cap C_0)\setminus \psi_d^{-1}\left(\bigcup_{L\in\Lc_T}\pi_{V_d(B)}^{-1}(L)\right)\right|\geq |A\cap C_0|-d^42^{2^{d+3}}.
\end{equation}
From Lemma \ref{R32}.iii,  each element in $S_0=\varphi_{d,B}((A\cap C_0)\setminus E_d(B))$ has at most $\delta_d(A,B)\leq d^2$ elements in its preimage so (\ref{E161}) leads to (\ref{E157}).

Now we show  i). Set $\Lc_1:=\{L\in\Oc_2(S):\;L\cap T=\emptyset\}$. On the one hand, $S$ is not collinear because if this is the case, then  $S_0$ is  also collinear and then $\pi_{V_d(B)}^{-1}(S_0)$ is contained in a hyperplane; however,   (\ref{E157})  yields $|S_0|>d^2$, and this would contradict (\ref{E154}). On the other hand, (\ref{E157}) implies $|S|>c_8(d)$. Thus we can apply Lemma \ref{R19} to $S$ and $T$, and we obtain that 
\begin{equation}
\label{E162}
|\Lc_1|\geq \frac{1}{2}|S|-c_9(d).
\end{equation}
From (\ref{E157}) and (\ref{E162}),
\begin{equation}
\label{E163}
|\Lc_1|\geq \frac{1}{2^2d^2}|A\cap C_0| -c_9(d)\geq \frac{1}{2^3d^2}|A\cap C_0|
\end{equation}
Proceeding as in the last part of Lemma \ref{R33},  it is concluded that  
\begin{equation*}
\eta_2(\eta_1(\Lc_1))\subseteq \Oc_{d,n}(A,B)
\end{equation*}
so
\begin{equation}
\label{E164}
\left|\Oc_{d,n}(A,B)\right|\geq |\eta_2(\eta_1(\Lc_1))|\geq \frac{1}{d^d}|\Lc_1|.
\end{equation}
Then i) is a straight consequence of  (\ref{E163}) and (\ref{E164}).

Finally, we show ii).  Set $S_1:=\varphi_{d,B}((A\setminus C_0)\setminus E_d(B))$, and notice that 
\begin{equation}
\label{E165}
|S_1|\leq |A\setminus C_0|\leq \frac{1}{2^2d^4}|A\cap C_0|.
\end{equation}
Denote by $\Lc_0$ the family of lines $L$ generated by $S_0$ such that $L\cap (T\cup S_1)=\emptyset$. For any $\ssb\in S_0$, write $\Lc_0(\ssb):=\{L\in\Lc_0:\;\ssb\in\Lc_0\}$. Since $S_0\subseteq\varphi_{d,B}(C_0\setminus E_d(B))\subseteq \pi_{V_d(B)}(\psi_d(C_0))$, we have that for each $\tb\in S_0\cap L$, there is $\zb\in \psi_d(C_0)\cap \pi^{-1}_{V_d(B)}(L)$ such that $\pi_{V_d(B)}(\zb)=\tb$. Then  (\ref{E154}) gives 
\begin{equation}
\label{E166}
|L\cap S_0|\leq |\pi_{V_d(B)}^{-1}(L\cap S_0)|\leq |\pi_{V_d(B)}^{-1}(L)\cap\psi_d(C_0)|\leq d^2.
\end{equation}
On the one hand, (\ref{E166}) implies that for each $L\in\Lc_0$,  there are at most $d^2$ elements $\ssb\in S_0$ such that $L\in\Lc_0(\ssb)$ so 
\begin{equation}
\label{E167}
|\Lc_0|\geq \frac{1}{d^2}\sum_{\ssb\in S_0}|\Lc_0(\ssb)|.
\end{equation}
On the the other hand, (\ref{E166}) implies that for each $\ssb\in S_0$, there are at least $\frac{|S_0|-1}{d^2}$ lines generated by $\ssb$ and other element of $S_0$; notice that at most $|S_1|+|T|$ of these lines pass through an element of $S_1\cup T$ so (\ref{E153}), (\ref{E157}) and (\ref{E165}) lead to
\begin{equation}
\label{E168}
|\Lc_0(\ssb)|\geq \frac{|S_0|-1}{d^2}-|S_1|-|T|\geq  \frac{1}{2^3d^4}|A\cap C_0|.
\end{equation}
From (\ref{E157}), (\ref{E167}) and (\ref{E168}), we can bound $|\Lc_0|$ below as follows
\begin{equation}
\label{E169}
|\Lc_0|\geq \frac{1}{d^2}\sum_{\ssb\in S_0}|\Lc_0(\ssb)|\geq \frac{1}{2^3d^6}|S_0||A\cap C_0|\geq \frac{1}{2^4d^8}|A\cap C_0|^2.
\end{equation}
We will show that 
\begin{equation}
\label{E170}
\eta_2(\eta_1(\Lc_0))\subseteq \Oc_{d,d^2}(A,B).
\end{equation}
For any $L\in\Lc_0$,  notice that $L\cap S=L\cap S_0$ insomuch as $L\cap (S_1\cup T)=\emptyset$; considering the preimages of $\pi_{V_d(B)}$, we get that $\eta_1(L)\cap \psi_d(A)=\eta_1(L)\cap\pi_{V_d(B)}^{-1}(S_0)$.  As in (\ref{E166}), 
\begin{equation*}
|\eta_2(\eta_1(L))\cap A|=|\eta_1(L)\cap \psi_d(A)|\leq d^2,
\end{equation*}
and this proves  (\ref{E170}). From Remark \ref{R31}, $\eta_1$ is injective.  From Lemma \ref{R11}, for each $C\in\Cc_{\leq d}$, there are at most $d^d$ hyperplanes in $\Real^{\binom{d+2}{2}-1}$ such that $C=\psi_d^{-1}(H)$ so $|\eta_2^{-1}(C)|\leq d^d$ for all $C\in \Cc_{\leq d}$. From this and (\ref{E170}),
\begin{equation}
\label{E171}
\left|\Oc_{d,d^2}(A,B)\right|\geq |\eta_2(\eta_1(\Lc_0))|\geq \frac{1}{d^d}|\Lc_0|.
\end{equation}
Then ii) is a consequence of (\ref{E169}) and (\ref{E171}).
\end{proof}
\section{Proofs of the main results}
We conclude the proofs of the main results in this section.
\begin{proof}\emph{(Theorem \ref{R5}). }
From Theorem \ref{R2}, the claim holds for $d=1$. From \cite[Thm 1.3]{Hu2}, the statement is true for $d=2$. Thus we assume that $d\geq 3$ from now on.  We show that $c_2=d^{d+9}2^{2^{4d+16}}\max\{c_8(d),c_9(d),1\}$
 and $c_3=\frac{1}{d^{2d+6}2^{(2d)^{4d+16}}(2^{d+3})!}$ satisfy the desired properties. Set $c_{10}:=\binom{\frac{3d^2-3d+4}{2}}{\binom{d+2}{2}-3}$ and $c_{11}:= \frac{(d2^{2^{3d+8}})^dc_{10}}{c_3}$. For all $B\in\Nc_d(A)$,  Lemma \ref{R32}.iii implies that
\begin{equation*}
\delta_d(A,B)+\binom{d+2}{2}-3=\delta_d(A,B)+|B|\leq \delta_d(A,B)+|D_d(B)|\leq d^2
\end{equation*} 
 so  $2\delta_d(A,B)+|D_d(B)|\leq \frac{3d^2-3d+4}{2}$.  Therefore
\begin{equation}
\label{E172}
\Oc_{d,2\delta_{d}(A,B)+|D_d(B)|}(A,B)\subseteq \Oc_{d,\frac{3d^2-3d+4}{2}}(A,B).
\end{equation}
For all $C\in \Oc_{d,\frac{3d^2-3d+4}{2}}(A)$, we have that $|A\cap C|\leq \frac{3d^2-3d+4}{2}$; therefore there are at most $c_{10}$ subsets $B\in\Nc_d(A)$ such that $C\in \Oc_{d,\frac{3d^2-3d+4}{2}}(A,B)$, and this yields 
\begin{align}
\label{E173}
\left|\Oc_{d,\frac{3d^2-3d+4}{2}}(A)\right|&\geq \left|\bigcup_{B\in\Nc_d(A)}\Oc_{d,\frac{3d^2-3d+4}{2}}(A,B)\right|\nonumber\\
&\geq \frac{1}{c_{10}}\sum_{B\in \Nc_d(A)}\left|\Oc_{d,\frac{3d^2-3d+4}{2}}(A,B)\right|.
\end{align}

Since $A$ is not contained in an element of $\Cc_d$, then $A$ is not contained in an element of $\Cc_{\leq d}$ by Remark \ref{R14}. The conclusion of the proof is divided into two cases. 
\begin{enumerate}
\item[i)]Assume that $A$ is $d$-regular. On the one hand, Lemma \ref{R24} gives 
\begin{equation}
\label{E174}
    |\Nc_d(A)|\geq |\Nc_d(A,\emptyset,\Real^2)|\geq \frac{1}{2^{d+3}!}|A|^{\binom{d+2}{2}-3}.
    \end{equation}
Lemma \ref{R33} and (\ref{E172}) imply that for all $B\in\Nc_d(A,\emptyset,\Real^2)$,
\begin{equation}
\label{E175}
\left| \Oc_{d,\frac{3d^2-3d+4}{2}}(A,B)\right|\geq \frac{1}{d^{d+2}2^4}|A|.
\end{equation}
   Then (\ref{E173}), (\ref{E174}) and (\ref{E175}) yield
     \begin{equation*}
     \left|\Oc_{d,\frac{3d^2-3d+4}{2}}(A)\right|\geq \frac{1}{c_{10}d^{d+2}2^4(2^{d+3}!)}|A|^{\binom{d+2}{2}-2}.
     \end{equation*}
   \item[ii)]Assume that $A$ is not  $d$-regular. Then there is $C\in \Cc_{\leq d}$ such that  $|A\cap C|\geq  \frac{1}{2^{2^{3d+8}}}|A|$. Moreover, since $C$ has at most $d$ irreducible components, there are $f\in[0,d-1]$ and $C_0\in \Cc_{d-f}$ irreducible such that 
   \begin{equation}
   \label{E176}
   |A\cap C_0|\geq  \frac{1}{d2^{2^{3d+8}}}|A|.
\end{equation} 
Denote by $\Rc$ the family of subsets $B_0\in\Pc_{\binom{f+2}{2}}(A\setminus C_0)$ such that $B_0$ is not contained in an element of $\Cc_{\leq f}$. Lemma  \ref{R16} gives 
  \begin{equation}
  \label{E177}
|\Rc|\geq \frac{1}{2^{\binom{f+2}{2}-1}}|A\setminus C_0|.
\end{equation}
For all $B_0\in\Rc$, Lemma \ref{R28} leads to
\begin{equation}
\label{E178}
    |\Nc_d(A,B_0,C_0)|\geq \frac{1}{2^{d+3}!}|A\cap C_0|^{\binom{d+2}{2}-3-\binom{f+2}{2}}\geq  \frac{1}{2^{d+3}!}|A\cap C_0|^{d-2}.
    \end{equation}
    We have two subcases.
    \begin{enumerate}
    \item[$\star$]Assume that $|A\setminus C_0|\leq \frac{1}{2^2d^4}|A\cap C_0|$. Hence Lemma \ref{R34}.ii implies that for all $B_0\in\Rc$ and  $B\in\Nc_d(A,B_0,C_0)$, 
\begin{equation}
\label{E179}
  \left|\Oc_{d,\frac{3d^2-3d+4}{2}}(A,B)\right|\geq \left|\Oc_{d,d^2}(A,B)\right|\geq \frac{1}{2^4d^{d+8}}|A\cap C_0|^2.
\end{equation}
Then
\begin{align*}
\left|\Oc_{d,\frac{3d^2-3d+4}{2}}(A)\right|&\geq \frac{1}{c_{10}}\sum_{B\in \Nc_d(A)}\left|\Oc_{d,\frac{3d^2-3d+4}{2}}(A,B)\right|&\Big(\text{by (\ref{E173})}\Big)\\
&\geq \frac{1}{c_{10}}\sum_{B_0\in\Rc}\sum_{B\in \Nc_d(A,B_0,C_0)}\left|\Oc_{d,\frac{3d^2-3d+4}{2}}(A,B)\right|\\
&\geq \frac{1}{c_{11}}|A\cap C_0|^d,&\Big(\text{by (\ref{E178}), (\ref{E179})}\Big)
\end{align*} 
 and the claim holds by (\ref{E176}).
   \item[$\star$]Assume that $|A\setminus C_0|> \frac{1}{2^2d^4}|A\cap C_0|$ so that (\ref{E177}) leads to
\begin{equation}
\label{E180}
|\Rc|\geq \frac{1}{2^{\binom{f+2}{2}-1}}  |A\setminus C_0|> \frac{1}{2^{\binom{f+2}{2}+1}d^4}|A\cap C_0|.
\end{equation} 
From Lemma \ref{R34}.i, we have that  for all $B_0\in\Rc$ and  $B\in\Nc_d(A,B_0,C_0)$, 
\begin{equation}
\label{E181}
  \left|\Oc_{d,\frac{3d^2-3d+4}{2}}(A,B)\right|\geq \frac{1}{2^3d^{d+2}}|A\cap C_0|.
\end{equation}  
 Thus  
\begin{align*}
\left|\Oc_{d,\frac{3d^2-3d+4}{2}}(A)\right|&\geq \frac{1}{c_{10}}\sum_{B\in \Nc_d(A)}\left|\Oc_{d,\frac{3d^2-3d+4}{2}}(A,B)\right|&\Big(\text{by (\ref{E173})}\Big)\\
&\geq \frac{1}{c_{10}}\sum_{B_0\in\Rc}\sum_{B\in \Nc_d(A,B_0,C_0)}\left|\Oc_{d,\frac{3d^2-3d+4}{2}}(A,B)\right|\\
&\geq \frac{1}{c_{11}}|A\cap C_0|^d,&\Big(\text{by (\ref{E178}), (\ref{E180}), (\ref{E181})}\Big)
\end{align*} 
 and the claim holds by (\ref{E176}). 
      \end{enumerate}
\end{enumerate}
Therefore in any case $c_2$ and $c_3$ work.
\end{proof}
We prove Theorem \ref{R6}.
\begin{proof}
\emph{(Theorem \ref{R6}). }Take a line $L$ in $\Real^2$, $B_0\in\Pc_{\binom{d+1}{2}}(\Real^2\setminus L)$ such that there is no element of $\Cc_{\leq d-1}$ which contains $B_0$, and $B_1\in \Pc_{m-\binom{d+1}{2}}( L)$. Set $A:=B_0\cup B_1$.  Theorem \ref{R9} implies that for any $C\in\Cc_{\leq d}$ such that $L$ is not a component of $C$, 
\begin{equation}
\label{E182}
 | L\cap C|\leq d.
\end{equation}

Now assume that $A$ is   contained in a curve $C\in\Cc_{\leq d}$. Since $d<m-\binom{d+1}{2}=|A\cap L|$,  (\ref{E182}) implies that $C$ contains $L$. However,  if $C$ contains $L$, then $B_0=A\setminus L\subseteq C\setminus L$, and therefore there is an element of $\Cc_{\leq d-1}$ which contains $B_0$ contradicting the assumption. Thus  $A$ is not contained in a curve $C\in\Cc_{\leq d}$, and this shows i).

Set
\begin{equation*}
\eta:\Oc_{d,\frac{3d^2-3d+4}{2}}(A)\longrightarrow \Pc_d(A\cap L), \qquad \eta(C)=A\cap L\cap C.
\end{equation*}
We show that $\eta$ is well defined. Since $|A\cap L|=m-\binom{d+1}{2}>\frac{3d^2-3d+4}{2}-\binom{d+1}{2}$, we have that $C$ cannot contain $L$ for any  $C\in \Oc_{d,\frac{3d^2-3d+4}{2}}(A)$; hence (\ref{E182}) yields that $|A\cap L\cap C|\leq |L\cap C|\leq d$. On the other hand, for any $C\in \Oc_{d,\frac{3d^2-3d+4}{2}}(A)$, $C$ is determined by $A$ so $|A\cap C|\geq \binom{d+2}{2}-1$, and hence
\begin{equation*}
|A\cap L\cap C|=|A\cap C|-|(A\cap C)\setminus L|\geq \binom{d+2}{2}-1-\binom{d+1}{2}=d.
\end{equation*}
  Thus $\eta$ is well defined. Finally, $\eta$ in injective. Indeed, take $C_1,C_2\in \Oc_{d,\frac{3d^2-3d+4}{2}}(A)$ such that $\eta(C_1)=\eta(C_2)$. Since $C_1$ and $C_2$ are determined by $A$, note  that $|A\cap C_1|, |A\cap C_2|\geq \binom{d+2}{2}-1$ so  $|(A\cap  C_1)\setminus L|, |(A\cap  C_2)\setminus L|\geq \binom{d+1}{2}$. Thereby $(A\cap  C_1)\setminus L=(A\cap  C_2)\setminus L=B_0$, and the equality $\eta(C_1)=\eta(C_2)$ implies that $A\cap C_1=A\cap C_2$. Since $C_1$ and $C_2$ are determined by $A$, the previous equality yields $C_1=C_2$. Finally, because $\eta$ is injective, 
\begin{equation*}
\left|\Oc_{d,\frac{3d^2-3d+4}{2}}(A)\right|\leq \binom{|A\cap L|}{d}=\binom{|A|-\binom{d+1}{2}}{d}, 
\end{equation*}
 and this shows ii).
\end{proof}
We prove Theorem \ref{R7}.
\begin{proof}\emph{(Theorem \ref{R7}). }
  We prove that $c_4=d^{d+9}2^{2^{4d+16}}\max\{c_8(d),c_9(d),1\}$
 and $c_5=\frac{1}{d^{2d+6}2^{(2d)^{4d+16}}(2^{d+3})!}$ work. Write $c_{12}:=\binom{2n+1-\binom{d+2}{2}}{\binom{d+2}{2}-3}$ and $c_{13}:=\\c_{12}2^2d^{d+2}(2^{d+3})!$. Let $B\in\Nc_d(A)$. Take $\ab\in A$ such that $\delta_d(A,B)=\\\left|\varphi_{d,B}^{-1}\left(\varphi_{d,B}(\ab)\right)\right|$. Since $\pi_{V_d(B)}^{-1}(\varphi_{d,B}(\ab))$ is a $\binom{d+2}{2}-3$-flat in $\Real^{\binom{d+2}{2}-1}$, we have that for any $\ssb\in \psi_d(A)\setminus \pi_{V_d(B)}^{-1}(\varphi_{d,B}(\ab))$, the flat $H_{\ssb}$ generated  by $\pi_{V_d(B)}^{-1}(\varphi_{d,B}(\ab))$ and $\ssb$ is a hyperplane.  Then, by Lemma \ref{R13}, $\psi_d^{-1}(H_\ssb)$ is an element of $\Cc_d$ containing $\varphi_{d,B}^{-1}\left(\varphi_{d,B}(\ab)\right)\cup D_d(B)$. Now, since $A$ is not contained in an element of $\Cc_{d}$, Remark \ref{R14} implies that $A$ is not contained in an element of $\Cc_{\leq d}$. Therefore, by Remark \ref{R12}.ii, $\psi_d(A)$ is not contained in a hyperplane of $\Real^{\binom{d+2}{2}-1}$. Thus, since $\pi_{V_d(B)}^{-1}(\varphi_{d,B}(\ab))$ is a $\binom{d+2}{2}-3$-flat, there are  $\ssb_1,\ssb_2\in  \psi_d(A)\setminus \pi_{V_d(B)}^{-1}(\varphi_{d,B}(\ab))$ and $\psi_d^{-1}(H_{\ssb_1})\neq \psi_d^{-1}(H_{\ssb_2})$. Since $\varphi_{d,B}^{-1}\left(\varphi_{d,B}(\ab)\right)\cup D_d(B)$ is contained in $\psi_d^{-1}(H_{\ssb_1})$ and $\psi_d^{-1}(H_{\ssb_2})$ and they are in $\Cc_d$, the assumption on $n$ yields that $|\varphi_{d,B}^{-1}\left(\varphi_{d,B}(\ab)\right)\cup D_d(B)|<n$, and hence 
 \begin{equation}
 \label{E183}
 \delta_d(A,B)+|B|\leq \delta_d(A,B)+|D_d(B)|=|\varphi_{d,B}^{-1}\left(\varphi_{d,B}(\ab)\right)\cup D_d(B)|<n.
 \end{equation}
 From (\ref{E183}), note that   $2\delta_d(A,B)+|D_d(B)|\leq 2n+1-\binom{d+2}{2}$, and hence
\begin{equation}
\label{E184}
\Oc_{d,2\delta_{d}(A,B)+|D_d(B)|}(A,B)\subseteq \Oc_{d, 2n+1-\binom{d+2}{2}}(A,B).
\end{equation}
For all $C\in \Oc_{d,2n+1-\binom{d+2}{2}}(A)$, we have that $|A\cap C|\leq 2n+1-\binom{d+2}{2}$; thus there are at most $c_{12}$ subsets $B\in\Nc_d(A)$ such that $C\in \Oc_{d,2n+1-\binom{d+2}{2}}(A,B)$, and this yields 
\begin{align}
\label{E185}
\left|\Oc_{d,2n+1-\binom{d+2}{2}}(A)\right|&\geq \left|\bigcup_{B\in\Nc_d(A)}\Oc_{d,2n+1-\binom{d+2}{2}}(A,B)\right|\nonumber\\
&\geq \frac{1}{c_{12}}\sum_{B\in \Nc_d(A)}\left|\Oc_{d,2n+1-\binom{d+2}{2}}(A,B)\right|.
\end{align}
If $A$ is $d$-regular, then we proceed exactly as in Case i) of  Theorem \ref{R5} to conclude that 
\begin{equation*}
     \left|\Oc_{d,2n+1-\binom{d+2}{2}}(A)\right|\geq \frac{1}{c_{12}d^{d+2}2^4(2^{d+3}!)}|A|^{\binom{d+2}{2}-2}.
     \end{equation*}

From now on, we assume that $A$  is not regular. This means that  there is $C\in \Cc_{\leq d}$ such that  $|A\cap C|\geq  \frac{1}{2^{2^{3d+8}}}|A|$. In so far as, $C$ has at most $d$ irreducible components, there are $f\in[0,d-1]$ and $C_0\in \Cc_{d-f}$ such that 
   \begin{equation}
   \label{E186}
   |A\cap C_0|\geq  \frac{1}{d2^{2^{3d+8}}}|A|.
\end{equation} 
Notice that $f=0$; otherwise, $f>0$ so  for any $B\in P_n(A\cap C_0)$ and any curve $C_1\in\Cc_{f}$, we get a curve $C_0\cup C_1\in\Cc_{\leq d}$ such that $B\subseteq C_0\cup C_1$ contradicting the assumption about $n$. Then, since $f=0$,  Lemma \ref{R28} imples  that for all $\bb\in A\setminus C_0$, 
\begin{equation}
\label{E187}
    |\Nc_d(A,\{\bb\},C_0)|\geq \frac{1}{2^{d+3}!}|A\cap C_0|^{\binom{d+2}{2}-3-\binom{f+2}{2}}=\frac{1}{2^{d+3}!}|A\cap C_0|^{\binom{d+2}{2}-4}.
    \end{equation}
    From Lemma \ref{R34}.i, we have that  for all $\bb\in A\setminus C_0$ and  $B\in\Nc_d(A,\{\bb\},C_0)$, 
\begin{equation}
\label{E188}
  \left|\Oc_{d,2n+1-\binom{d+2}{2}}(A,B)\right|\geq \frac{1}{2^3d^{d+2}}|A\cap C_0|.
\end{equation}  
 Hence
\begin{align*}
\left|\Oc_{d,2n+1-\binom{d+2}{2}}(A)\right|&\geq \frac{1}{c_{12}}\sum_{B\in \Nc_d(A)}\left|\Oc_{d,2n+1-\binom{d+2}{2}}(A,B)\right|&\Big(\text{by (\ref{E185})}\Big)\\
&\geq \frac{1}{c_{12}}\sum_{\bb\in A\setminus C_0}\sum_{B\in \Nc_d(A,\{\bb\},C_0)}\left|\Oc_{d,2n+1-\binom{d+2}{2}}(A,B)\right|\\
&\geq \frac{1}{c_{13}}|A\cap C_0|^{\binom{d+2}{2}-3},&\Big(\text{by (\ref{E187}), (\ref{E188})}\Big)
\end{align*} 
 and the claim holds by (\ref{E186}). 
    \end{proof}
    
 Finally, we complete the proof of Theorem \ref{R8}.
    \begin{proof}\emph{(Theorem \ref{R8}). } Let $C_0\in\Cc_d$ be  irreducible and $H$ be a hyperplane in $\Real^{\binom{d+2}{2}-1}$ such that $C_0=\psi_d^{-1}(H)$. Choose  $\ab_0\in \Real^2\setminus C_0$. We construct recursively a set $S\in\Pc_{m-1}(\psi_d(\Real^2)\cap H)$ such that   $\dim R=|R|-1$ for all $R\in\Pc(S)$ with $\dim R<\binom{d+2}{2}-2$. Take $\ssb_1\in\psi_d(\Real^2)\cap H$ and write $S_1:=\{\ssb_1\}$. Now assume that for some $i\in[1,m-2]$, we have constructed a set $S_i\in\Pc_i(\psi_d(\Real^2)\cap H)$ such that   $\dim R=|R|-1$ for all $R\in\Pc(S_i)$ with $\dim R<\binom{d+2}{2}-2$. Let $\Fc_i$ be the collection of all flats $F$ generated by the subsets of $S_i$ such that $\dim F<\binom{d+2}{2}-2$. Since $S_i$ is finite, $\Fc_i$ is finite. On the other hand, for each $F\in\Fc_i$, there exists a hyperplane $G$ in $\Real^{\binom{d+2}{2}-1}$ such that $G\neq H$ and $F\subseteq G\cap H$; in particular,  $C_0$ is not a component of $\psi_d^{-1}(G)$. Applying Theorem \ref{R9} to the curves $\psi_d^{-1}(G)$ and $\psi_d^{-1}(H)=C_0$, we have that $|\psi_d^{-1}(G)\cap \psi_d^{-1}(H)|\leq d^2$ and thus
    \begin{equation}
    \label{E189}
    |F\cap (\psi_d(\Real^2)\cap H)|\leq  |G\cap (\psi_d(\Real^2)\cap H)|=|\psi_d^{-1}(G)\cap \psi_d^{-1}(H)|\leq d^2.
    \end{equation}
    From (\ref{E189}), we have that $\left(\bigcup_{F\in\Fc_i} F\right)\cap \left(\psi_d(\Real^2)\cap H\right)$ is finite. Since $\psi_d(C_0)\subseteq \psi_d(\Real^2)\cap H$ is not finite,  $(\psi_d(\Real^2)\cap H)\setminus \bigcup_{F\in\Fc_i} F\neq\emptyset$ and we choose $\ssb_{i+1}$ in this difference. Make $S_{i+1}:=S_i\cup\{\ssb_{i+1}\}$, and notice that $S_{i+1}\subseteq \psi_d(\Real^2)\cap H$,   $|S_{i+1}|=i+1$ and $\dim R=|R|-1$ for all $R\in\Pc(S_{i+1})$ with $\dim R<\binom{d+2}{2}-2$ (the last property because $\ssb_{i+1}\not\in \bigcup_{F\in\Fc_i} F$). In this way we construct $S_2,S_3,\ldots,S_{m-1}$ and   $S:=S_{m-1}$ has the desired properties. Set  $A:=\{\ab_0\}\cup \psi_d^{-1}(S)$.
    
     Theorem \ref{R9} implies that for any $C\in\Cc_{\leq d}$ such that $C_0\nsubseteq C$, 
\begin{equation}
\label{E190}
 | C_0\cap C|\leq d^2.
\end{equation}
Since $d^2<m-1=|A\cap C_0|$,  if  $A$ is   contained in a curve $C\in\Cc_{\leq d}$, then $C_0$ is contained in $C$ by (\ref{E190}). Nevertheless, in so far as $C_0\in\Cc_d$ and $C\in\Cc_{\leq d}$, we have that $C=C_0$ but this is impossible since $\ab_0\not\in C_0$. This proves i).

Take $B\in \Pc_n(A)$ and write $B_0:=B\cap C_0$. We claim that
\begin{align}
\label{E191}
\dim \psi_d(B_0)&\geq \left\{ \begin{array}{ll}
\binom{d+2}{2}-2& \text{if $B=B_0$}\\
&\\
 \binom{d+2}{2}-3 & \text{if $B\neq B_0$}.\end{array} \right.
\end{align}
 Indeed,  if $B=B_0$, then $|\psi_d(B_0)|=|B|\geq \binom{d+2}{2}-1$; thus, if $\dim \psi_d(B_0)<\binom{d+2}{2}-2$, note that
 \begin{equation*}
 \dim \psi_d(B_0)<\binom{d+2}{2}-2\leq |\psi_d(B_0)|-1,
 \end{equation*}
which is impossible  by the construction of $S=\psi_d(A\cap C_0)$. If $B\neq B_0$, then $B=B_0\cup \{\ab_0\}$ so $|\psi_d(B_0)|=|B|-1\geq \binom{d+2}{2}-2$; if $\dim \psi_d(B_0)<\binom{d+2}{2}-3$, then 
 \begin{equation*}
 \dim \psi_d(B_0)<\binom{d+2}{2}-3\leq |\psi_d(B_0)|-1,
 \end{equation*}
which is impossible  by the construction of $S$, and this proves (\ref{E191}). Now, if $B=B_0$, then (\ref{E191}) implies that there is at most one hyperplane which contains $\psi_d(B)$; hence, by Remark \ref{R12}.ii, there is at most one element $C\in\Cc_{\leq d}$ such that $B\subseteq C$.  If $B\neq B_0$, then (\ref{E191}) leads to $\dim \psi_d(B_0)\geq \binom{d+2}{2}-3$. Since $\ab_0\not\in C$, we have that $\psi_d(\ab_0)\not\in H$ , and since $H\supseteq \psi_d(B_0)$, we conclude that $\dim \psi_d(B)\geq 1+\dim \psi_d(B_0)\geq \binom{d+2}{2}-2$. As in the previous case, we conclude that there is at most one element $C\in\Cc_{\leq d}$ such that $B\subseteq C$, and this completes the proof of ii).  

Finally we show iii). Set
\begin{equation*}
\eta:\Oc_{d,2n+1-\binom{d+2}{2}}(A)\longrightarrow \Pc_{\binom{d+2}{2}-2}(A\cap C_0), \qquad \eta(C)=A\cap C_0\cap C.
\end{equation*}
We show that $\eta$ is well defined. Take  $C\in \Oc_{d,2n+1-\binom{d+2}{2}}(A)$. Since $|A\cap C_0|=m-1>2n+1-\binom{d+2}{2}-1$, we have that $C$ cannot contain $C_0$. Since $C$ is determined by $A$, we have that $|A\cap C|\geq \binom{d+2}{2}-1$. We prove by contradiction that 
\begin{equation}
\label{E192}
|A\cap C_0\cap C|\leq \binom{d+2}{2}-2.
\end{equation}
  Assume that 
\begin{equation}
\label{E193}
|\psi_d(A\cap C_0\cap C)|=|A\cap C_0\cap C|\geq \binom{d+2}{2}-1. 
\end{equation}
By the construction of $S$, (\ref{E193}) yields that
\begin{equation}
\label{E194}
\dim \psi_d(A\cap C_0\cap C)\geq \binom{d+2}{2}-2=\dim \psi_d(C_0);
\end{equation}
 however, since $C$ does not contain $C_0$, we have that $\psi_d(C_0\cap C)$ is contained in two different hyperplanes of $\Real^{\binom{d+2}{2}-1}$ and therefore 
\begin{equation*}
\dim \psi_d(A\cap C_0\cap C)\leq \dim \psi_d(C_0\cap C)\leq \binom{d+2}{2}-3,
\end{equation*}
 which contradicts (\ref{E194}) and proves (\ref{E192}). On the other hand, since $C$ is determined by $A$, $|A\cap C|\geq \binom{d+2}{2}-1$ and thus 
 \begin{equation}
 \label{E195}
 |A\cap C_0\cap C|\geq |A\cap C|-|\{\ab_0\}|\geq \binom{d+2}{2}-2
 \end{equation}
 From (\ref{E192}) and (\ref{E195}), $\eta$ is well defined; furthermore these inequalities force that
 \begin{equation*}
  |A\cap C_0\cap C|= |A\cap C|-|\{\ab_0\}|
 \end{equation*}
so  $\ab_0\in C$ for all $C\in \Oc_{d,2n+1-\binom{d+2}{2}}(A)$. Hence,  for any $C_1,C_2\in \Oc_{d,2n+1-\binom{d+2}{2}}(A)$ such that $A\cap C_0\cap C_1=A\cap C_0\cap C_2$, we get  that $A\cap C_1=A\cap C_2$. In so far as $C_1$ and $C_2$ are determined by $A$, we conclude that $C_1=C_2$. This yields that $\eta$ is injective so  
 \begin{equation*}
| \Oc_{d,2n+1-\binom{d+2}{2}}(A)|\leq \binom{|A\cap C_0|}{\binom{d+2}{2}-2}= \binom{|A|-1}{\binom{d+2}{2}-2}
\end{equation*}
concluding the proof of iii).
    \end{proof}
 
\end{document}